\documentclass[12pt,english,oneside,a4paper]{amsart}

\usepackage[english]{babel}
\usepackage[utf8]{inputenc}
\usepackage[T1]{fontenc}
\usepackage[bottom=3cm,a4paper]{geometry}
\usepackage[hidelinks]{hyperref}
\usepackage{amsmath,amsfonts,amssymb,mathtools,amsthm}
\usepackage{listings,verbatim,fancyvrb,xspace,enumerate}
\usepackage{xcolor,graphicx,ytableau,tikz}
\usepackage[all]{xy}
\usepackage{fancyhdr}

\newtheorem{theorem}{Theorem}[subsection]
\newtheorem*{theorem*}{Theorem}

\newtheorem{corollary}[theorem]{Corollary}
\newtheorem*{corollary*}{Corollary}
\newtheorem{lemma}[theorem]{Lemma}
\newtheorem{proposition}[theorem]{Proposition}
\theoremstyle{definition}
\newtheorem{example}[theorem]{Example}
\newtheorem*{example*}{Example}
\newtheorem*{examples*}{Examples}
\newtheorem{remark}[theorem]{Remark}
\newtheorem{definition}[theorem]{Definition}
\newtheorem{question}{Question}
\newtheorem{conjecture}{Conjecture}
\newtheorem*{ack}{Acknowledgements}
\newtheorem*{con}{Conventions}
\newsavebox{\eqbox}
\newenvironment{longequation*} {\begin{lrbox}{\eqbox}$} {$\end{lrbox}\begin{equation*}\resizebox{\linewidth}{!}{\ensuremath{\displaystyle\usebox{\eqbox}}}\end{equation*}}

\makeatletter
\newcommand*\bigcdot{\mathpalette\bigcdot@{.5}}
\newcommand*\bigcdot@[2]{\mathbin{\vcenter{\hbox{\scalebox{#2}{\(\m@th#1\bullet\)}}}}}

\makeatother

\def\bydef{\coloneqq}

\DeclarePairedDelimiter{\abs}{\lvert}{\rvert}
\DeclarePairedDelimiter{\Set}{\lbrace}{\rbrace}
\newcommand{\suchthat}{\hskip.5ex\color{black!66}\middle/\color{black}\hskip1ex}

\DeclareMathOperator{\diff}{d}

\DeclareMathOperator{\Flag}{Flag}

\DeclareMathOperator{\rank}{rank}

\DeclareMathOperator{\codim}{codim}

\DeclareMathOperator{\ord}{ord}
\DeclareMathOperator{\Co}{Col}
\DeclareMathOperator{\Div}{div}
\DeclareMathOperator{\Pic}{Pic}
\DeclareMathOperator{\inv}{inv}

\DeclareMathOperator{\Diag}{Diag}

\DeclareMathOperator{\Sing}{Sing}

\DeclareMathOperator{\Ext}{\mathchoice
  {\scalebox{1.2}{$\mathsf{\Lambda}$}}
  {\scalebox{1.1}{$\mathsf{\Lambda}$}}
  {\scalebox{0.8}{$\mathsf{\Lambda}$}}
  {\scalebox{.65}{$\mathsf{\Lambda}$}}
}

\newcommand{\C}{\mathbb{C}}
\renewcommand{\O}{\mathcal{O}} 
\renewcommand{\P}{\mathbf{P}} 
\newcommand{\R}{\mathbb{R}}
\newcommand{\N}{\mathbb{N}}
\newcommand{\Z}{\mathbb{Z}}
\newcommand{\B}{\mathbb{B}}
\newcommand{\kk}{\mathbf{k}}
\let\mathbi\boldsymbol

\author{Antoine Etesse}
\email{antoine.etesse@univ-amu.fr}
\address{Aix Marseille Univ, CNRS, Centrale Marseille, I2M, Marseille, France}
\title
[Geometric generalized Wronskians. Applications in intermediate hyperbolicity and foliation theory.]
{Geometric generalized Wronskians. Applications in intermediate hyperbolicity and foliation theory.}

\begin{document}
\sloppy

\begin{abstract}
In this paper, we introduce a sub-family of the usual generalized Wronskians, that we call geometric generalized Wronskians. It is well-known that one can test linear dependance of holomorphic functions (of several variables) via the identical vanishing of generalized Wronskians. We show that such a statement remains valid if one tests the identical vanishing only on geometric generalized Wronskians. 

It turns out that geometric generalized Wronskians allow to define intrinsic objects on projective varieties polarized with an ample line bundle: in this setting, the lack of existence of global functions is compensated by global sections of powers of the fixed ample line bundle. Geometric generalized Wronskians are precisely defined so that their local evaluations on such global sections globalize up to a positive twist by the ample line bundle.

We then give three applications of the construction of geometric generalized Wronskians: one in intermediate hyperbolicity, and two in foliation theory. In intermediate hyperbolicity, we show the algebraic degeneracy of holomorphic maps from \(\C^{p}\) to a Fermat hypersurface in \(\P^{N}\) of degree \(\delta > (N+1)(N-p)\): this interpolates between two well-known results, namely for \(p=1\) (first proved via Nevanlinna theory) and \(p=N-1\) (in which case the Fermat hypersurface is of general type). The first application in foliation theory provides a criterion for algebraic integrability of leaves of foliations: our criterion is not optimal in view of current knowledges, but has the advantage of having an elementary proof. Our second application deals with positivity properties of adjoint line bundles of the form \(K_{\mathcal{F}} + L \), where \(K_{\mathcal{F}}\) is the canonical bundle of a regular foliation \(\mathcal{F}\) on a smooth projective variety \(X\), and where \(L\) is an ample line bundle on \(X\).

\end{abstract}

\maketitle

\section*{Introduction}
Wronskians first appeared in the litterature with the work of Wronski \cite{hoene1812refutation} in 1812: being given \(f_{0}, \dotsc, f_{m}\) \((m+1)\) holomorphic functions on the complex line \(\C\) (or any open subset of it), their \textsl{Wronskian} \(W(f_{0}, \dotsc, f_{m})\) is the determinant of the square matrix formed with their multi-derivatives:
\[
W(f_{0}, \dotsc, f_{m})
\bydef
\det
\begin{pmatrix}
f_{0} & \cdot & \cdot & f_{m}
\\
f_{0}' & \cdot & \cdot & f_{m}'
\\ 
\cdot  & & & \cdot
\\
\cdot & & & \cdot
\\
 f_{0}^{(m)} & \cdot & \cdot & f_{m}^{(m)}
\end{pmatrix}.
\]
This definition, specific to one variable functions, has been widely used in the litterature since it was first introduced: see e.g \cite{surveyonevariable} for a survey of its applications in various areas of mathematics. (Note that their definition of ``generalized Wronskians`` differs from ours). A fundamental property of Wronskians is the following:
\begin{center}
\((f_{0}, \dotsc, f_{m})\) are linearly independant over \(\C\) if and only if their Wronskian \(W(f_{0}, \dotsc, f_{m})\) does not vanish identically.
\end{center}
This was pointed out by Peano in \cite{Peano}, and first proved by Bôcher in \cite{Bocher}.

It is only later that the definition of Wronskians was extended to the multi-variables case in \cite{Ostro}.
The definition goes as follows (see also Section~\ref{sse: g W}). Let \(p\in \N_{\geq 1}\), and let \(f_{0}, \dotsc, f_{m}\) be \((m+1)\) holomorphic functions on \(\C^{p}\)(or an open subset of it). Consider a set of words \(\mathcal{U}=\Set{\overline{u_{1}}, \dotsc, \overline{u_{m}}}\), where each word \(\overline{u_{i}}\) is written with the alphabet \(\Set{1, \dotsc, p}\). Define the following determinant:
\[
W_{\mathcal{U}}(f_{0}, \dotsc, f_{m})
\bydef
\det
\begin{pmatrix}
f_{0} & \cdot & \cdot & f_{m}
\\
\partial_{\overline{u_{1}}} f_{0} & \cdot & \cdot & \partial_{\overline{u_{1}}} f_{m}
\\ 
\cdot & & & \cdot
\\
\cdot & & & \cdot
\\
\partial_{\overline{u_{m}}} f_{0} & \cdot & \cdot & \partial_{\overline{u_{m}}} f_{m}
\end{pmatrix},
\]
where by definition
\[
\partial_{\overline{u}}
=
\frac{\partial^{\alpha_{1}(\overline{u})+ \dotsb + \alpha_{p}(\overline{u})}}
{\partial z_{1}^{\alpha_{1}(\overline{u})} \dotsb \partial z_{p}^{\alpha_{p}(\overline{u})}},
\]
with \(\overline{u}=1^{\alpha_{1}(\overline{u})}\dotsb p^{\alpha_{p}(\overline{u})}\).
This is called a \textsl{pure generalized Wronskian} if and only if the length \(\ell(\overline{u_{i}})\) of the word \(\overline{u_{i}}\) is less or equal than \(i\) for every \(1 \leq i \leq m\). A \textsl{generalized Wronskian} is then simply a linear combination of pure generalized Wronskians. It was already hinted at in \cite{Ostro} that generalized Wronskians must statisfy the same fundamental property than Wronskians, namely:
\begin{center}
\((f_{0}, \dotsc, f_{m})\) are linearly independant over \(\C\) if and only if one of their pure generalized Wronskians \(W_{\mathcal{U}}(f_{0}, \dotsc, f_{m})\) does not vanish identically.
\end{center}
This result was proved and used in \cite{Roth} in the context of Diophantine approximation, and reproved later in \cite{Ber} in the context of Nevanlinna Theory and algebraic dependance of real numbers. Note that each proof provided in the previously quoted papers uses an induction on the number of variables.

Recently, in respectively \cite{luza2018extactic} and \cite{brotbek2017hyperbolicity}, Wronskians in one variable were used to construct global jet differentials (of \(1\)-germs) on projective varieties: see Section~\ref{ssse: jets differentials} for  definitions of jet differentials for \(p\)-germs, \(p\geq 1\), or \cite{brotbek2017hyperbolicity} for the case \(p=1\). Their construction reads as follows.
\begin{proposition}[\cite{brotbek2017hyperbolicity}, \cite{luza2018extactic}]
\label{prop: Brot}
Let \(s_{0}, \dotsc, s_{m} \) be global sections of a line bundle \(L \) on a projective variety \(X\). The Wronskian operator \( W \) induces a global section \(W(s_{0}, \dotsc, s_{m}) \) of \(E_{1,m,\frac{m(m+1)}{2}}X \otimes L^{m+1}\), defined locally on a trivializing open set \(U \) of \(L\) as follows:
\[
W(s_{0}, \dotsc, s_{m})(\gamma)
\overset{loc}{=}
W(s_{0, U}\circ \gamma, \dotsc, s_{m,U} \circ \gamma)(0),
\]
where \( \gamma: (\C,0) \to (X,x) \) is a \(1\)-germ passing through \(x \in U\), and where the notation \(s_{i,U}\) indicates that the section \(s_{i}\) is written under a trivialization of the line bundle \(L\) on the open set \(U\).
\end{proposition}
This is the starting point of this paper, as our main goal is to generalize the previous proposition in the setting of \(p\)-germs, \(p>1\), using generalized Wronskians. There is however an obstacle in this higher dimensional case: whereas the compatibility condition to define a global object \(W(s_{0}, \dotsc, s_{m})\)  is automatically satisfied for the one variable Wronskian W (this gives rise to the positive twist \(L^{m+1}\)), this is no longer true for generalized Wronskian. This lead us to introduce \textsl{geometric generalized Wronskians} (see Section~\ref{sse: gg W}): those are exactly the generalized Wronskians satisfying the wanted compatibility condition. The analogue of Proposition \ref{prop: Brot}  reads as follows (see Section~\ref{sse: g W} for definitions of order and weight, and see Section~\ref{ssse: jets differentials} for definitions of relevant jet differentials bundles):
\begin{proposition}
\label{prop: constr}
Let \(s_{0}, \dotsc, s_{m} \) be global sections of a line bundle \(L \). Let \( W \) be a non-zero geometric generalized Wronskian of order \(k\) and weight \(w\). The operator \( W \) induces a global section \(W(s_{0}, \dotsc, s_{m}) \) of \(E_{p,k,w}X \otimes L^{m+1}\), defined locally on a trivializing open set \(U \) of \(L\) as follows:
\[
W(s_{0}, \dotsc, s_{m})(\gamma)
\overset{loc}{=}
W(s_{0, U}\circ \gamma, \dotsc, s_{m,U} \circ \gamma)(0),
\]
where \( \gamma: (\C^{p},0) \to (X,x) \) is a \(p\)-germ passing through \(x \in U\), and where the notation \(s_{i,U}\) indicates that the section \(s_{i}\) is written under a trivialization of the line bundle \(L\) on the open set \(U\).
\end{proposition}
In view of this, and the fundamental property of generalized Wronskians, the following question is important:
\begin{center}
``Are geometric generalized Wronskians enough to ensure linear dependance of functions? ``
\end{center}
The answer is positive, and it is our first Main Theorem (labeled Theorem \ref{thm: geometric generalized W} in Section~\ref{sse: proof linear independance}):
\begin{theorem}[Main Theorem I]
\label{thm: mainthm intro}
Let \((f_{0}, \dotsc, f_{m})\) be \((m+1)\) holomorphic functions on \(\C^{p}\). They are linearly independant if and only there exists a geometric generalized Wronskian \(W\) such that \(W(f_{0}, \dotsc, f_{m})\) does not vanish identically.
\end{theorem}
As a matter of fact, we prove a stronger result: one can test the non-vanishing condition only on \textsl{pure geometric generalized Wronskians}. Those form a particular subfamily of pure generalized Wronskians: see Section~\ref{sse: pgg W} for definitions, and Section~\ref{sse: proof linear independance} for the proof.

In the second part of this paper, we use our construction (detailed in Proposition \ref{prop: constr}) and our Main Theorem \ref{thm: mainthm intro} in order to study intermediate hyperbolicity properties of Fermat hypersurfaces. For an introduction to hyperbolic spaces see e.g. \cite{kob13}, and for intermediate notions of hyperbolicity, namely \textsl{\(p\)-analytic hyperbolicity} see e.g. \cite{Santa}. For our purposes, it is enough to know that intermediate hyperbolicity is concerned with family of \textsl{entire curves} (an entire curve in a variety \(X\) is a non-constant holomorphic map from the complex line \(\C\) to \(X\)). Our application in intermediate hyperbolicity (labeled Theorem \ref{thm: hyperbolicity} in Section~\ref{sse: hyperbolicity}) reads as follows:
\begin{theorem}[Main Theorem II]
\label{thm: mainthm 2}
Let \(H_{\boldsymbol{\lambda}}=\Set{\lambda_{0}X_{0}^{\delta} + \dotsb + \lambda_{N}X_{N}^{\delta}=0} \subset \P^{N} \) be a Fermat hypersurface of degree \(\delta \in \N_{\geq1}\) in the projective space \(\P^{N}\), \(N \geq 1\), where \(\boldsymbol{\lambda}=[\lambda_{0}, \dotsc, \lambda_{N}] \in \P^{N}\). 
Let \(1 \leq p \leq N-1\), and suppose that \(\delta > (N+1)(N-p)\). Then any non-degenerate holomorphic map
\[
f\colon \C^{p} \to H_{\boldsymbol{\lambda}}
\]
is algebraically degenerate, i.e. its image lies in an hypersurface of \(H_{\boldsymbol{\lambda}}\).
\end{theorem}
This result is well-known for the extremal cases. For \(p=N-1\), this follows from the fact that \(H_{\boldsymbol{\lambda}}\) is of general type. For \(p=1\), this was first proved via Nevanlinna theory: see e.g \cite{kob13}[p.144-145]. Therefore, our result interpolates between these two situations. 

In the third and last part of this paper, we provide two applications in foliation theory.
The first application has starting point the following application of Proposition \ref{prop: Brot}, which was pointed out to me by Erwan Rousseau, following an observation of Jorge Vitorio Pereira (see Section \ref{se: foliation intro} for relevant definitions).
\begin{proposition}
\label{prop: fol curves}
Let \(X\) be a normal projective variety, and let \(\mathcal{F}\) be a foliation by curves on \(X\). 
Suppose that \(\mathcal{F}\) is not algebraically integrable, i.e. that a general leaf of \(\mathcal{F}\) is not algebraic. Then the canonical bundle of the foliation \(K_{\mathcal{F}}\) is pseudo-effective.
\end{proposition}
Note that pseudo-effectiveness of canonical bundles of foliations was thoroughly studied in the litterature, and one has in particular the following difficult result: 
if the canonical bundle of a foliation on a smooth projective variety is not pseudo-effective, then the foliation is uniruled (see \cite{CampanaPaun}). Since a uniruled foliation of rank \(1\) is always algebraically integrable, this implies in particular the above Proposition \ref{prop: fol curves}.

However, in Proposition \ref{prop: fol curves}, one should see the canonical bundle \(K_{\mathcal{F}}\) as the cotangent bundle \(\Omega_{\mathcal{F}}\). This is indeed the right terminology to obtain a result valid for any foliation, which constitutes a result owing to works of Bost \cite{Bostfol}, Bogomolov and McQuillan \cite{bogomolov2016rational}, Campana and Paun \cite{CampanaPaun}, and Druel \cite{Druel}:
\begin{theorem}[\cite{Bostfol}, \cite{bogomolov2016rational}, \cite{CampanaPaun}, \cite{Druel}]
\label{thm: CPBD}
Let \(X\) be a normal projective variety, and let \(\mathcal{F}\) be a foliation of rank \(p\) on \(X\). 
Suppose that \(\mathcal{F}\) is not algebraically integrable. Then the cotangent bundle of the foliation \(\Omega_{\mathcal{F}}\) is pseudo-effective.
\end{theorem}
In this statement, the definition of pseudo-effectiveness is the following: the (reflexive) coherent sheaf \(\Omega_{\mathcal{F}} \bydef T_{\mathcal{F}}^{*}\) (\(T_{\mathcal{F}}\) is the involutive saturated coherent subsheaf of \(T_{X}\bydef \Omega_{X}^{*}\) defining the foliation \(\mathcal{F}\), see Definition \ref{def: foliation}) is pseudo-effective if and only if for any \(c \in \N_{\geq 1}\), there exists \(m, n \in \N_{\geq 1}\) satisfying \(m > nc\) such that the following holds
\[
H^{0}(X, S^{[m]}\Omega_{\mathcal{F}} \otimes L^{n}) \neq \Set{0},
\]
where \(S^{[m]} \Omega_{\mathcal{F}} \bydef (S^{m} \Omega_{F})^{**}\) is the reflexive hull of \(S^{m} \Omega_{\mathcal{F}}\), see Definition \ref{def: reflexive}.
Using geometric generalized Wronskians, we were able to obtain a weak version of the previous Theorem \ref{thm: CPBD}. Our first application in foliation theory (labeled Theorem \ref{thm: foliation} in Section~\ref{sse: foliation integrability}) reads then as follows:
\begin{theorem}[Main Theorem III]
\label{thm: mainthm 3}
Let \(X\) be a normal projective variety, \(\mathcal{F}\) a  foliation of rank \(p \geq 1\) on \(X\) and \(L \to X\) an ample line bundle. Suppose that \(\mathcal{F}\) is not algebraically integrable. Then for any \(c \in \N_{\geq 1}\), there exists \(m, n \in \N_{\geq 1}\) satisfying \(m > nc\) such that the following holds
\[
H^{0}(X, \Omega_{\mathcal{F}}^{[m]} \otimes L^{n}) \neq \Set{0}.
\]
\end{theorem}
Therefore, the conclusion of our Theorem \ref{thm: mainthm 3} concerns the non-vanishing of (reflexive hulls of) twisted tensor products of the cotangent bundle of the foliation, whereas in Theorem \ref{thm: CPBD}, it concerns the non-vanishing of (reflexive hulls of) twisted symmetric powers of the cotangent bundle of the foliation.

Our second application in foliation theory concerns positivity of adjoint line bundles of the form
\[
K_{\mathcal{F}} + L
\]
on a foliated smooth projective variety \((X, \mathcal{F})\), where \(L\) is an ample line bundle and \(\mathcal{F}\) is a regular foliation on \(X\). A somewhat simplified statement reads as follows:
\begin{theorem}[Main Theorem IV]
\label{thm: mainthm 4}
Let \(X\) be a smooth projective variety, \(\mathcal{F}\) a regular foliation of rank \(p \geq 1\)  and \(L \to X\) an ample line bundle. 
Suppose that there exists a constant \(K >0\) such that the following inequality holds outside a countable union of points:
\[
\varepsilon(L;x) \geq K.
\]
Then the line bundle \(K_{\mathcal{F}}+\frac{(p+1)}{K}L\) is nef.
\end{theorem}
In this statement, the number \(\varepsilon(L;x)\) denotes the Seshadri constant of \(L\) at \(x\) (see e.g. \cite{Laz1}[Section 5]).
For instance, if one takes \(\mathcal{F}=TX\) the trivial foliation, one obtains informations on the usual adjoint line bundles of the form \(K_{X}+mL\), where \(m \in \N_{\geq 1}\) and \(L\) is an ample line bundle. Note however that our result gives nothing new in this regard: powerful tools such as Kodaïra type vanishing theorem (see \cite{Laz1}[Theorem 4.3.1]) allows to obtain much better conclusions (see \cite{Laz1}[Proposition 5.1.19]). It may however provide indications that some vanishing results involving the canonical bundle \(K_{X}\) might also work if one replaces it by the canonical bundle of a foliation \(\mathcal{F}\). We refer to Section \ref{sse: adjoint} for a more detailed discussion.
\subsubsection*{The paper is organized as follows.}\mbox{}

Section~\ref{sse: g W} is devoted to recalling the construction of generalized Wronskians as well as the proof of their fundamental property, namely that they allow to detect linear independance of holomorphic functions. 

Section~\ref{sse: gg W} introduces \textsl{geometric generalized Wronskians}, and is twofold. In Section~\ref{ssse: jets differentials}, we define jet differentials for \(p\)-germs, \(p\in \N_{\geq 1}\), on a complex manifold: these are the right geometric objects to define differential operators acting on \(p\)-germs, e.g. generalized Wronskians of functions. In Section \ref{ssse: gg W}, we then define geometric generalized Wronskian: these are the generalized Wronskians which, when evaluated on \textsl{sections of line bundles},  act effectively on \(p\)-germs (not every generalized Wronskian does so!).

Section~\ref{sse: pgg W} exhibits a particular family of geometric generalized Wronskians, that we call \textsl{pure geometric generalized Wronskians}.

Section~\ref{sse: proof linear independance} is devoted to the proof of our Main Theorem \ref{thm: mainthm intro}. It is actually a more precise version since we show the following: pure geometric generalized Wronskians are enough to read linear independance of holomorphic functions. This is the main technical part of the paper.

Section~\ref{sse: gg W and repr} describes a filtration of geometric generalized Wronskians that will be used in our first application in folitation theory Theorem \ref{thm: mainthm 3}.

Section~\ref{sse: hyperbolicity} details the proof of our application in intermediate hyperbolicity Theorem \ref{thm: mainthm 2}. We show the algebraic degeneracy of non-degenerate holomorphic maps from \(\C^{p}\), \(p \in \N_{\geq 1}\), to Fermat hypersurfaces in \(\P^{N}\) of degree \(\delta > (N+1)(N-p)\).

Section~\ref{sse: foliation integrability} details the proof of our first application in foliation theory Theorem \ref{thm: mainthm 3}, which gives a criterion for algebraic integrability of leaves of foliations.

Section~\ref{sse: adjoint} contains our second application in foliation theory Theorem \ref{thm: mainthm 4}, where some specific geometric generalized Wronskians are used to study the positivity of adjoint line bundles of the form \(K_{\mathcal{F}} + mL\), where \(\mathcal{F}\) is a regular foliation, \(L\) an ample line bundle,  and \(m\) a positive integer.

 Appendix~\ref{appendix: formula} gives two formulas repeatedly used in the text: one is a Leibniz rule to compute partial multi-derivatives of functions, and one is a formula to compute partial multi-derivatives of composition of maps.

 Appendix~\ref{appendix: vanishing} details the proof of a vanishing result needed for our application in intermediate hyperbolicity Theorem \ref{thm: mainthm 2}.

\begin{con}
Throughout this paper, a \textsl{variety} is a reduced and irreducible scheme separated and of finite type over the field of complex numbers \(\C\).
\end{con}

\section{Geometric Generalized Wronskians}

\subsection{Generalized Wronskians}
\label{sse: g W}
In this section, we introduce  generalized Wronskians as it was defined in e.g. \cite{Roth}. We prove their fundamental property, following the scheme of proof given in \cite{Wronskians}: note that it differs from the one in given in \cite{Roth} or \cite{Ber} as it does not proceed by induction. Every detail is given since the proof of our first Main Theorem \ref{thm: mainthm intro} relies on it.

Let \(p \in \N_{\geq 1} \) be an integer. Denote \(\mathcal{W}_{p}\) the set of words \textsl{written in the lexicographic order} with the alphabet \(\Set{1, \dotsc, p}\). Accordingly, any word \( \overline{u} \in \mathcal{W}_{p} \) writes uniquely
\[
\overline{u}
=
1^{\alpha_{1}(\overline{u})} \dotsb p^{\alpha_{p}(\overline{u})},
\]
where by definition, \( \alpha_{i}(\overline{u}) \) is the number of occurences of the letter \(i\) in the word \(\overline{u}\). The \textsl{length} of the word \(\overline{u}\) is the number \( \ell(\overline{u})=\alpha_{1}(\overline{u}) + \dotsb + \alpha_{p}(\overline{u})\). Given a natural number \(k \in \N_{\geq 1}\), denote \(\mathcal{W}_{p, \leq k}\) the set of words of length less or equal than \(k\), and \(\mathcal{W}_{p,k}\) the set of words of length \(k\).
\begin{definition}[Size, order, weight and caracteristic exponent of sets of words]
\label{def: order}

Let \(\mathcal{U}\) be a finite set of words in \(\mathcal{W}_{p}\). 

The \textsl{size} \(m\) of the set \(\mathcal{U}\) is its cardinal, i.e. \(m=\abs{\mathcal{U}}\).

The \textsl{order} of \(\mathcal{U}\) is the maximal length \(k\) of a word belonging to \(\mathcal{U}\). 

The \textsl{weight} of \(\mathcal{U}\) is the integer 
 \(
 w(\mathcal{U})
 \bydef
 \sum\limits_{\overline{u} \in \mathcal{U}} \ell(\overline{u})
 \).
 
 The \textsl{caracteristic exponent} \(\boldsymbol{\beta}(\mathcal{U}) \in \N^{p}\) of the set \(\mathcal{U}\) is defined as follows:
\[
\boldsymbol{\beta}(\mathcal{U})
\bydef
(\beta_{1}, \dotsc, \beta_{p}),
\]
where \(\beta_{i}=\sum\limits_{\overline{u} \in \mathcal{U}} \alpha_{i}(\overline{u})\). In particular, one has that
\[
|\mathbi{\beta}(\mathcal{U})|
= 
w(\mathcal{U}).
\]
\end{definition}
\begin{definition}[Admissible sets]
We say that \(\mathcal{U}\) is an \textsl{admissible set} if and only if there exists an ordering \(\Set{\overline{u_{1}}, \dotsc, \overline{u_{m}}}\) of the words in \(\mathcal{U}\) such that for every \( 1 \leq i \leq m \), \( \ell(\overline{u_{i}}) \leq i \).
\end{definition}
\begin{remark}
\label{rmk: order}
The following ordering will be adopted. Start by listing every word of length \(1\) in \(\mathcal{U}\), according to the lexicographic order on \(\Set{1, \dotsc, p}\). Then, list every word of length \(2\) in \(\mathcal{U}\), according to the lexicographic order on  \(\Set{1, \dotsc, p}^{2}\). Do this until all the words of \(\mathcal{U}\) are exhausted.
\end{remark}
To any set of words \( \mathcal{U} \) of size \(m\) is associated the linear functional on \(\Ext^{m+1}\mathcal{O}_{\C^{p}} \)  defined as follows:
\[
\left(W_{\mathcal{U}}\right)_{x}(f_{0}, \dotsc, f_{m})
\bydef
\det
\begin{pmatrix}
f_{0} & \cdot & \cdot & f_{m}
\\
\partial_{\overline{u_{1}}} f_{0} & \cdot & \cdot & \partial_{\overline{u_{1}}} f_{m}
\\ 
\cdot & & & \cdot
\\
\cdot & & & \cdot
\\
\partial_{\overline{u_{m}}} f_{0} & \cdot & \cdot & \partial_{\overline{u_{m}}} f_{m}
\end{pmatrix}(x),
\]
where the \(f_{i}\)'s belong to \( \mathcal{O}_{\C^{p},x} \) (i.e. are germs of holomorphic functions around \(x\)), and where we recall that
\[
\partial_{\overline{u}}
=
\frac{\partial^{\alpha_{1}(\overline{u})+ \dotsb + \alpha_{p}(\overline{u})}}
{\partial z_{1}^{\alpha_{1}(\overline{u})} \dotsb \partial z_{p}^{\alpha_{p}(\overline{u})}}.
\]
\begin{definition}[Generalized Wronskians]
Let \(m \in \N_{\geq 1}\) be an integer, and let \( \mathcal{U} \) be an admissible set of size \(m \).
The functional \(W_{\mathcal{U}}\) is called a \textsl{pure generalized Wronskian}, and its \textsl{size} is by definition equal to \(m\).

Denote by \(\mathcal{W}(m) \) the finite dimensional \( \C \)-vector space spanned by pure generalized Wronskians \(\left(W_{\mathcal{U}}\right)_{\abs{\mathcal{U}}=m} \). 
An element \(W\) of \(\mathcal{W}(m)\) is called a \textsl{generalized Wronskian} of size \(m\). Its \textsl{order} is by definition the least order of the admissible sets appearing in the writing of \(W\).

A generalized Wronskian \(W\) is called \textsl{unmixed of weight \(w\)} if and only if the admissible sets appearing in its writing all have the same weight \(w\).
\end{definition}
\begin{remark}
By definition, a pure generalized Wronskian is always unmixed.
\end{remark}
There is a natural action of the group of biholomorphisms of \( (\C^{p}, 0) \) on the vector space \(\mathcal{W}(m) \) obtained as follows. Let \( \varphi \) be a biholomorphism of \( (\C^{p},0) \), let \(x \in \C^{p}\), and let \(f_{0}, \dotsc, f_{m} \) be germs of holomorphic functions around \(x\). Define
\[
  (\varphi \cdot W)_{x}(f_{0}, \dotsc, f_{m})
  \bydef 
  (W_{\mathcal{U}})_{0}(f_{0}(x+\varphi(\cdot)), \dotsc, f_{m}(x+\varphi(\cdot))).
 \]
Using the formula Lemma \ref{lemma: compo} to compute multi-derivatives of compositions of holomorphic maps, as well as the multi-linearity of the determinant, one easily sees that the functional \( \varphi \cdot W \) can be expressed as a linear combination of the functionals \( \left(W_{\mathcal{U}}\right)_{\abs{\mathcal{U}}=m}\), with constant coefficients depending on the multi-derivatives (evaluated in \(0\)) of the coordinates functions of the biholomorphism \( \varphi \): the functional \( \varphi \cdot W \) is indeed in \(\mathcal{W}(m)\), and the action is well-defined.  In particular, considering only the linear biholomorphisms of \( (\C^{p},0) \), one obtains a representation of the linear group \( GL_{p}(\C) \). Hence the natural following question:
\begin{question}
\label{q: irr rep}
What are the irreducible components of the previous representation?
\end{question}

The fundamental property of generalized Wronskians is the following:
\begin{theorem}[\cite{Roth}]
\label{thm: generalized W}
Let \(f_{0}, \dotsc, f_{m} \) be \((m+1)\) be holomorphic functions on \(\C^{p}\). They are linearly independant if and only if there exists an admissible set \( \mathcal{U} \) of size \( m \) such that \( W_{\mathcal{U}}(f_{0}, \dotsc, f_{m}) \not\equiv 0 \).
\end{theorem}
We detail the proof, following \cite{Wronskians}.
\begin{proof}
If the holomorphic  functions \(f_{0}, \dotsc, f_{m}\) are linearly dependant, it follows from the multi-linearity and alternating property of the determinant that, for every admissible set \(\mathcal{U} \) of size \(m\), the function \(W_{\mathcal{U}}(f_{0}, \dotsc, f_{m})\) vanishes identically. Thus, the non-trivial part amounts to proving that if \((f_{0}, \dotsc, f_{m})\) is a free family, then there exists an admissible set \(\mathcal{U} \) of size \(m\) such that \(W_{\mathcal{U}}(f_{0}, \dotsc, f_{m}) \not\equiv 0\). It is enough to prove this in the local case, i.e. one can suppose without loss of generality that \(f_{0}, \dotsc, f_{m}\) are series in \(\C[[z_{1}, \dotsc, z_{p}]]\). Put the lexicographic order on the exponents of the power series, and call the \textsl{order} of a power serie $s$, denoted \(\ord(s)\), the least exponent in the lexicographic order appearing effectively in the writing of the serie \(s\).
We start with the following elementary lemma
\begin{lemma}
\label{lemma: serie}
There exists \(t_{0}, \dotsc, t_{m} \in \C[[z_{1}, \dotsc, z_{p}]]\) with distinct order,  and an invertible matrix \(A \in GL_{m+1}(\C)\) such that
\[
(t_{0}, \dotsc, t_{m})=(f_{0}, \dotsc, f_{m})\cdot A.
\]
\end{lemma}
\begin{proof}[proof of Lemma \ref{lemma: serie}]
By applying a suitable permutation matrix \(A\), one can beforehand suppose that 
\[
\ord(f_{0}) \leq \ord(f_{1}) \leq \dotsc \leq \ord(f_{m}).
\]
If the series already have distinct orders, there is nothing to prove. Otherwise, let \(i_{0}\) be the least integer \(i\) such that \(\ord(f_{i})=\ord(f_{i+1})\). Then, as the family \((f_{i}, f_{i+1})\) is free, there exists \(\lambda \in \C^{*}\) such that 
\[
\infty > \ord(f_{i+1}-\lambda f_{i}) > \ord(f_{i+1})=\ord(f_{i}).
\]
Let \(A\) be the matrix of transvection associated to the previous transformation, so that
\[
(f_{0}, \dotsc, f_{i}, f_{i+1}-\lambda f_{i}, f_{i+2}, \dotsc, f_{m})
=
(f_{0}, \dotsc, f_{m}) \cdot A.
\]
Apply now the same reasoning to the new free family of series 
\[
(\tilde{f_{1}}, \dotsc, \tilde{f_{m}})=(f_{1}, \dotsc, f_{i}, f_{i+1}-\lambda f_{i}, f_{i+2}, \dotsc, f_{m}).
\]
Reorder them, and consider the least integer \(i_{1}\) such that the order of two series is the same (if it exists, otherwise we are done). As \(i_{1} > i_{0}\), an obvious induction proves the result.
\end{proof}
One immediately checks that, since the coefficients of the invertible matrix \(A\) are constants, one has the equality 
\[
W_{\mathcal{U}}(t_{0}, \dotsc, t_{m})
=
\det(A) W_{\mathcal{U}}(f_{0}, \dotsc, f_{m})
\]
for every admissible set \(\mathcal{U}  \) of size \(m\). Therefore, with Lemma \ref{lemma: serie}, one can suppose without loss of generality that the series \(f_{0}, \dotsc, f_{m}\) have distinct orders 
\[
\boldsymbol{\alpha}_{i}=(\alpha_{1,i}, \dotsc, \alpha_{p,i}).
\]

The key observation is then the following: if for an admissible set \(\mathcal{U}\) of size \(m\), the holomorphic function 
\(
W_{\mathcal{U}}(\boldsymbol{z}^{\boldsymbol{\alpha}_{0}}, \dotsc, \boldsymbol{z}^{\boldsymbol{\alpha}_{m}})
\)
is non-zero, then so is \(W_{\mathcal{U}}(f_{0}, \dotsc, f_{m})\). (Recall the multi-index notation \(\mathbi{z}^{\mathbi{\alpha}} \bydef z_{1}^{\alpha_{1}} \dotsc z_{p}^{\alpha_{p}}\)).
Indeed, by multi-linearity of the determinant, one sees that in this situation, the monomial term with least order appearing in the writing of \(W_{\mathcal{U}}(f_{0}, \dotsc, f_{m})\) comes from 
\(W_{\mathcal{U}}(\boldsymbol{z}^{\boldsymbol{\alpha}_{1}}, \dotsc, \boldsymbol{z}^{\boldsymbol{\alpha_{m}}})\); more specifically, this least order is then equal to 
\[
\sum\limits_{i=0}^{m} \mathbi{\alpha_{i}} - \big(\sum\limits_{\overline{u} \in \mathcal{U}} \alpha_{1}(\overline{u}), \dotsc, \sum\limits_{\overline{u} \in \mathcal{U}} \alpha_{p}(\overline{u})\big).
\]

To conclude the proof, it is now enough to prove that if the \(p\)-uples  \( \left(\boldsymbol{\alpha_{i}}\right)_{0 \leq i \leq m} \) are distinct, then there exists an admissible set \(\mathcal{U}\) of size \(m\) such that 
\[
W_{\mathcal{U}}(\boldsymbol{z}^{\boldsymbol{\alpha}_{0}}, \dotsc, \boldsymbol{z}^{\boldsymbol{\alpha}_{m}})
\not\equiv
0.
\]
For a proof of this statement, we refer to \cite{Wronskians}. In Section \ref{sse: proof linear independance}, we will actually prove a stronger statement.
\end{proof}

\subsection{Geometric Generalized Wronskians}
\label{sse: gg W}
In this section, we first recall the construction of jet spaces and jet differentials bundles for \(p\)-germs into a complex variety \( X\), \( p \in \N_{\geq 1} \). This is the right context to define generalized Wronskians acting on sections of line bundles on the complex variety \( X \): those generalized Wronskians are called \textsl{geometric}. Their definition and properties is the object of the second part of this section.

\subsubsection{Jet spaces and jet differentials bundles for \(p\)-germs}
\label{ssse: jets differentials}
Let \(X\) be a complex manifold of dimension \(N\), and let \(p \in \N_{\geq 1}\). For \(k \in \N_{\geq 1}\), define the bundle \(J_{p,k}X\) of \(k\)-jets of \(p\)-germs of holomorphic maps \(\gamma\colon (\C^{p},0) \rightarrow X\) on the complex manifold $X$ as follows. Consider an atlas \((U_{i}, \varphi_{i})_{i \in I}\) of \(X\), and for \(i \in I\), consider the (trivial) bundle on \(U_{i}\) whose fiber over \(x \in U_{i}\) is the \(\C\)-vector space of dimension \(N\times |\mathcal{W}_{p,\leq k}|\)
\[
\left\{
\Big(\partial_{\overline{u}}(\varphi_{i} \circ \gamma)(0)\Big)_{\overline{u} \in \mathcal{W}_{p,\leq k}}
\
\big|
\
\gamma: (\C^{p},0) \rightarrow (X,x) \ \text{holomorphic \(p\)-germ} \
\right\}.
\]
Glue these trivial bundles 
\(U_{i} \times \C^{N\times |\mathcal{W}_{p,\leq k}|}\) via (the maps naturally induced by) the transition maps \(\varphi_{j} \circ \varphi_{i}^{-1}\) to obtain (up to isomorphism) the bundle \(J_{p,k}X\). The general formula to change charts involves higher order derivatives of the transition maps as soon as \(k>1\), and in particular, it does not preserve the structure of vector space of the fibers: \(J_{p,k}X\) is not a vector bundle for \(k>1\). For sake of notation, denote
\(
\partial_{p,k}\gamma
\)
the element in $(J_{p,k}X)_{x}$ defined by the holomorphic \(p\)-germ \(\gamma\colon (\C^{p},0) \to (X,x)\).
\begin{example}
In the case where \(k=1\), \(J_{p,1}X\) is \(p\) copies of the tangent bundle \(TX\). In the case where \(p=1\), one recovers jet spaces of curves.
\end{example}
 
The torus $(\C^{*})^{p}$ acts on the fibers of \(J_{p,k}X\) as follows
\[
 \boldsymbol{\lambda}
\cdot 
\Big(\partial_{\overline{u}}(\varphi_{i} \circ \gamma)\Big)_{\overline{u} \in \mathcal{W}_{p,\leq k}}
=
\Big(\lambda_{1}^{\alpha_{1}(\overline{u})}\dotsc \lambda_{p}^{\alpha_{p}(\overline{u})} \partial_{\overline{u}}(\varphi_{i} \circ \gamma)\Big)_{\overline{u} \in \mathcal{W}_{p,\leq k}},
\]
 where \(\boldsymbol{\lambda}=(\lambda_{1}, \dotsc, \lambda_{p}) \in (\C^{*})^{p}\).
Indeed, the element $\boldsymbol{\lambda}=(\lambda_{1}, \dotsc, \lambda_{p}) \in (\C^{*})^{p}$ induces a diagonal automorphism of $\C^{p}$  
\[
D_{\boldsymbol{\lambda}} 
\colon
(z_{1}, \dotsc, z_{p}) \mapsto (\lambda_{1}z_{1}, \dotsc, \lambda_{p}z_{p}),
\]
and one has the equality
\[
\boldsymbol{\lambda} \cdot \partial_{p,k} \gamma
=
 \left(\partial_{\overline{u}}(\varphi_{i} \circ \gamma \circ D_{\boldsymbol{\lambda}})\right)_{\overline{u} \in \mathcal{W}_{p,\leq k}}.
 \]
Therefore, the action commutes with a  change of chart, and is thus well defined on the bundle $J_{p,k}X$. More generally, the group of biholomorphisms of \( (\C^{p}, 0) \) acts on \(J_{p,k}X\) by setting
\[
\varphi \cdot \partial_{p,k} \gamma
\bydef
\partial_{p,k}(\gamma \circ \varphi),
\]
where \(\varphi\) is a biholomorphism of \((\C^{p},0)\) and \(\gamma: (\C^{p}, 0) \to X\) is a \(p\)-germ. Since one considers derivatives up to the order \(k\), the algebraic description of the previous action is encoded in the following complex Lie group
\[
\mathbb{G}_{p,k}
\bydef
\Big\{
\big(\sum \limits_{1 \leq \ell(\overline{u}) \leq k} a_{1,\overline{u}} \cdot \mathbi{z}^{\overline{u}}, \dotsc, \sum \limits_{1 \leq \ell(\overline{u}) \leq k} a_{p,\overline{u}} \cdot \mathbi{z}^{\overline{u}} \big)
\ \big| \ 
\big(a_{i,j}\big)_{1 \leq i, j \leq p} \in GL_{p}(\C)
\Big\},
\]
 in which the composition law is given by reducing modulo every monomial of degree strictly greater than \(k\), and where \(\mathbi{z}^{\overline{u}}\bydef z_{1}^{\alpha_{1}(\overline{u})} \dotsb z_{p}^{\alpha_{p}(\overline{u})}\).

Define the vector bundle \(E_{p,k,\boldsymbol{\beta}}X\) of partial jet differentials (of \(p\)-germs) of order \(k\) and multi-degrees \(\boldsymbol{\beta}=(\beta_{1}, \dotsc, \beta_{p}) \in \N^{p}\) as follows. Construct the bundle whose fiber over $x \in X$ is the vector space of complex valued polynomials $Q$ of muti-degrees $\boldsymbol{\beta}$ on the fiber $(J_{p,k}X)_{x}$, i.e. 
for any \(\boldsymbol{\lambda} \in \C^{p}\) and any \(p\)-germ \(\gamma\colon (\C^{p},0) \to (X,x) \), the polynomial \(Q\) satisfies the equality
\[
Q(\boldsymbol{\lambda} \cdot \partial_{p,k}\gamma)
=
\boldsymbol{\lambda}^{\boldsymbol{\beta}} Q(\partial_{p,k}\gamma),
\]
where we recall the usual multi-indexes notation \(\boldsymbol{\lambda}^{\boldsymbol{\beta}}=\lambda_{1}^{\beta_{1}}\dotsb \lambda_{p}^{\beta_{p}}\). The formula to compute multi-derivatives of compositions of maps Lemma \ref{lemma: compo} allows one to see that the structure of vector space of the fibers is preserved under a change of chart: the bundle \(E_{p,k,\boldsymbol{\beta}}X\) is indeed a vector bundle.

Let \(w \in \N_{\geq 1}\), and gather the multi-indexes \(\boldsymbol{\beta} \in \N^{p}\) of same weight \( \abs{\boldsymbol{\beta}}=w\) in order to define the vector bundle of partial jet differentials of order \(k\) and weight \(w\)
\[
E_{p,k,w}X
\bydef
\underset{|\boldsymbol{\beta}|=w}{\bigoplus} E_{p,k, \boldsymbol{\beta}}X.
\]
The group of biholomorhisms of \((\C^{p},0)\) acts in a natural fashion on \(E_{p,k,w}X\) by acting on \(J_{p,k}X\):
\[
(\varphi \cdot Q)(\partial_{p,k} \gamma)
\bydef
Q(\varphi \cdot \partial_{p,k}\gamma),
\]
where \(Q\) is a partial jet differentials of order \(k\) and weight \(w\), and \(\varphi\) is a biholomorphism of \((\C^{p},0)\). Indeed, observe that one has the following equality for any \(\lambda \in \C^{*}\) and any \(\varphi\) biholomorphism of \((\C^{p},0)\):
\[
(\lambda I_{p} \circ \varphi) \cdot \partial_{p,k} \gamma
 =
 (\varphi \circ \lambda I_{p}) \cdot \partial_{p,k} \gamma,
\]
which follows for instance from the formula Lemma \ref{lemma: compo}. It then readily implies that if \(Q \in E_{p,k,w}X\), then so does \(\varphi \cdot Q\) for any \(\varphi\) biholomorphism of \((\C^{p},0)\).
\begin{remark}
\label{rmk: lin}
Note that \( E_{p,k, \boldsymbol{\beta}}X \) is \emph{not} stable under the action of biholomorphisms, unless \(\boldsymbol{\beta}=(0, \dotsc, 0)\). 
\end{remark}
One can define particular subbundles of \(E_{p,k,w}X\) by imposing a condition under the action of the biholomorphisms of \((\C^{p},0)\). Here, we introduce the following:
\begin{definition}[Linear and invariant subbundles]
The \textsl{linear subbbundle of \(E_{p,k,w}X\)} is defined as follows:
\[
E_{p,k,w}^{1}X
\bydef
\Set{
Q \in E_{p,k,w}X 
\ | \
\forall \ \varphi \ \text{biholomorphism of \((\C^{p},0)\)},
\
\varphi \cdot Q
=
\diff \varphi(0) \cdot Q
}.
\]
This is the subbundle on which the action of the biholomorphisms of \((\C^{p},0)\) reduces to the action of the linear group \(GL_{p}(\C)\). The bundle \(E_{p,k,w}^{1}X\) has itself a particular subbundle, the \textsl{invariant subbundle of \(E_{p,k,w}X\)}, defined as follows:
\[
E_{p,k,w}^{\inv}X
\bydef 
\Set{
Q \in E_{p,k,w}X 
\ | \
\forall \ \varphi \ \text{biholomorphism of \((\C^{p},0)\)},
\
\varphi \cdot Q
=
\det(\diff \varphi(0))^{\frac{w}{p}} Q
}.
\]
Note that this bundle is the zero-bundle if \(w\) is not divisible by \(p\).
\end{definition}
\begin{example}
In the case where \(p=1\), one recovers the usual invariant jet differentials bundle. Note that in this situation, one has the equality
\[
E_{1,k,w}^{\inv}X
=
E_{1,k,w}^{1}X.
\]
\end{example}

There is a natural filtration of the vector bundles \(E_{p,k,\boldsymbol{\beta}}X\) obtained as follows. Locally, an element of \(E_{p,k,\boldsymbol{\beta}}X\) is a linear combination of monomial terms of the following form
\[
\prod\limits_{i=1}^{N} (\partial_{\overline{u_{1}}}\gamma_{1})^{\tau_{1}} \dotsb (\partial_{\overline{u_{N}}}\gamma_{N})^{\tau_{N}}
\]
where \(\gamma=(\gamma_{1}, \dotsc, \gamma_{N}): (\C^{p},0) \to X\) is a \(p\)-germ (written in a fixed trivialization, with \(N=\dim(X)\)), and where the following  is satisfied:
\[
\forall \ 1 \leq j \leq N, \
\ell(\overline{u_{j}}) \leq k
\
\
\text{and}
\
\
\forall \ 1 \leq i \leq p, \
\sum\limits_{j=1}^{N} \tau_{j} \times \alpha_{i}(\overline{u_{j}})=\beta_{i}.
\]
Fix \(\overline{u} \in \mathcal{W}_{p,k}\) a word of length \(k\), and define the total degree \(\deg_{\overline{u}}\) with respect to the word \(\overline{u}\) as follows:
\[
 \deg_{\overline{u}}\left(\prod\limits_{i=1}^{N} (\partial_{\overline{u_{1}}}\gamma_{1})^{\tau_{1}} \dotsb (\partial_{\overline{u_{N}}}\gamma_{N})^{\tau_{N}}\right)
 \bydef
 \sum\limits_{i, \ \overline{u_{i}}=\overline{u}} \tau_{i}.
 \]
 It is extended to any element \(P\) of \(E_{p,k,\boldsymbol{\beta}}X\) by taking the maximum degree of the monomials appearing in the writing of \(P\). The degree \(\deg_{\overline{u}}\) does not depend on the trivialization, as one can readily check by changing charts and using the formula Lemma \ref{lemma: compo} to compute multi-derivatives of compositions of maps. It is thus well-defined on \(E_{p,k, \boldsymbol{\beta}}X\) and it induces the following filtration \(F_{\overline{u}}\) of \(E_{p,k, \boldsymbol{\beta}}X\):
 \[
 F_{\overline{u}}(E_{p,k, \boldsymbol{\beta}}X)\colon
 0 
 \subset 
 F_{\overline{u},0}(E_{p,k, \boldsymbol{\beta}}X)
  \subset 
 \dotsb 
 \subset 
 F_{\overline{u}, \lceil \frac{\abs{\boldsymbol{\beta}}}{k} \rceil}(E_{p,k, \boldsymbol{\beta}}X)
 =E_{p,k, \boldsymbol{\beta}}X
 \]
 where \(F_{\overline{u},d}(E_{p,k, \boldsymbol{\beta}}X)\) is the subvector bundle of elements of degree \(\deg_{\overline{u}}\leq d\). Note that for any \(0\leq d \leq \lceil \frac{\abs{\boldsymbol{\beta}}}{k} \rceil\), one has the isomorphism
 \[
 F_{\overline{u},d}(E_{p,k, \boldsymbol{\beta}}X)
 /
 F_{\overline{u},d-1}(E_{p,k, \boldsymbol{\beta}}X)
 \simeq
 S^{d}\Omega_{X} \otimes F_{\overline{u},0}(E_{p,k, \tilde{\boldsymbol{\beta}}(\overline{u},d)})
 \]
 where 
 \(\tilde{\boldsymbol{\beta}}(\overline{u}, d)
 \bydef 
 \boldsymbol{\beta} - d\times (\alpha_{1}(\overline{u}), \dotsc, \alpha_{p}(\overline{u}))\). 
 Inductively, this allows to obtain a filtration of \(E_{p,k,\boldsymbol{\beta}}X\) whose graded terms are as follows:
 \[
 \left(
 \bigotimes_{\overline{u}\in \mathcal{W}_{p,k}} S^{d(\overline{u})}\Omega_{X}
 \right)
 \bigotimes 
 \dotsb
 \bigotimes
 \left(
 \bigotimes_{\overline{u}\in \mathcal{W}_{p,1}} S^{d(\overline{u})} \Omega_{X}
 \right),
 \]
where the integers \(d(\overline{u})\) satisfy the following equality for any \(1 \leq j \leq p\)
\[
\sum\limits_{\overline{u} \in \mathcal{W}_{p, \leq k}} d(\overline{u})\alpha_{j}(\overline{u})
=
\beta_{j}.
\]
\begin{remark}
Let \(w\) be an integer, and consider all the \(p\)-uples \(\boldsymbol{\beta}\) of weight \(\abs{\boldsymbol{\beta}}=w\). Putting  the previous filtrations together, one obtains a filtration for the vector bundle \(E_{p,k,w}X\).
\end{remark} 
\subsubsection{Geometric generalized Wronskians}
\label{ssse: gg W}
Fix \( m \in \N_{\geq 1} \), and consider the following subvector space \(\mathcal{W}_{g}(m)\) of \(\mathcal{W}(m)\):
\[
\mathcal{W}_{g}(m)
\bydef
\Set{
W \in \mathcal{W}(m)
\ | \
 \forall g, f_{0}, \dotsc, f_{m} \in \mathcal{O}_{\C^{p}},
 \ 
 W(gf_{0}, \dotsc, gf_{m})
 =
 g^{m+1}W(f_{0}, \dotsc, f_{m})
 }.
 \]
\begin{definition}[Geometric generalized Wronskians]
\label{def: gg W}
A \textsl{geometric generalized Wronskian} is an element of \(\mathcal{W}_{g}(m)\), for some \( m \in \N_{\geq 1} \).
\end{definition}
Observe that the vector space of geometric generalized Wronskians of size \(m\) is stable under the action of the biholomorphisms of \((\C^{p},0)\). Hence the analogue of Question \ref{q: irr rep}:
\begin{question}
\label{q: irr rep bis}
What are the irreducible components of the representation (of \(GL_{p}(\C)\)) \(\mathcal{W}_{g}(m)\)?
\end{question}

Let \( X \) be a complex variety, equipped with a line bundle \( L \to X \). By very definition of geometric generalized Wronskians and partial jet differentials bundles, one has the following proposition:
\begin{proposition}
\label{prop: gg W germ}
Let \(s_{0}, \dotsc, s_{m} \) be global sections of the line bundle \(L \). Let \( W \) be an unmixed geometric generalized Wronskian of size \(m\), order \(k\) and weight \(w\). The operator \( W \) induces a global section \(W(s_{0}, \dotsc, s_{m}) \) of \(E_{p,k,w}X \otimes L^{m+1}\), defined locally on a trivializing open set \(U \) of \(L\) as follows:
\[
W(s_{0}, \dotsc, s_{m})(\gamma)
\overset{loc}{=}
W_{0}(s_{0, U}\circ \gamma, \dotsc, s_{m,U} \circ \gamma),
\]
where \( \gamma: (\C^{p},0) \to (X,x) \) is a \(p\)-germ passing through \(x \in U\), and where the notation \(s_{i,U}\) indicates that the section \(s_{i}\) is written under the trivialization of the line bundle \(L\) on the open set \(U\).
\end{proposition}
\begin{proof}
Compute that, for any \(\mathbi{\lambda}=(\lambda_{1}, \dotsc, \lambda_{p}) \in (\C^{*})^{p}\) and any full set \(\mathcal{U}\) of size \(m\) and order \(k\), one has the following equality
\[
D_{\mathbi{\lambda}} \cdot W_{\mathcal{U}}
=
\mathbi{\lambda}^{\mathbi{\beta}(\mathcal{U})} W_{\mathcal{U}},
\]
where we recall that \(D_{\mathbi{\lambda}}=\Diag(\lambda_{1}, \dotsc, \lambda_{p})\).
This implies in particular that for any \(p\)-germ \(\gamma\colon (\C^{p},0) \to (U,x)\), and any \(\mathbi{\lambda} \in (\C^{*})^{p}\) one has the following equality:
\[
(W_{\mathcal{U}})_{0}\big(s_{0,U}(\gamma \circ D_{\mathbi{\lambda}}), \dotsc, s_{m,U}(\gamma \circ D_{\mathbi{\lambda}})\big)
=
\mathbi{\lambda}^{\mathbi{\beta}(U)}  (W_{\mathcal{U}})_{0}\big(s_{0, U}\circ \gamma, \dotsc, s_{m,U} \circ \gamma\big).
\]
As \((W_{\mathcal{U}})_{0}\big(s_{0, U}\circ \gamma, \dotsc, s_{m,U} \circ \gamma\big)\) clearly writes as a polynomial in the coordinates of \(\partial_{p,k} \gamma\), this implies that \(W_{\mathcal{U}}(s_{0,U}, \dotsc, s_{m,U})\) defines a section of \(E_{p,k, \mathbi{\beta}(\mathcal{U})}U\).

Now, by hypothesis, the geometric generalized Wronskian \(W\) writes
\[
W=\sum\limits_{\mathcal{U}, w(\mathcal{U})=w} \lambda_{\mathcal{U}} W_{\mathcal{U}},
\]
where \(\lambda_{\mathcal{U}} \in \C\), and the above shows that \(W(s_{0,U}, \dotsc, s_{m,U})\) defines a section of
\[
E_{p,k,w}U=\bigoplus_{|\mathbi{\beta}|=w} E_{p,k, \mathbi{\beta}}U.
\]
The very definition of geometric generalized Wronskians allows then to glue the local sections \(W(s_{0,U}, \dotsc, s_{m,U})\) into a global section \(W(s_{0}, \dotsc, s_{m})\) of \(E_{p,k,w}X\otimes L^{m+1}\).
\end{proof}
\subsection{Pure geometric generalized Wronskians}
\label{sse: pgg W}
\begin{definition}[Pure geometric generalized Wronskians]
\label{def: pgg W}
Let \(\mathcal{U}\) be an admissible set. One says that \(\mathcal{U}\) is a \textsl{full set} if and only if the set \(\mathcal{U} \) satisfies the following property:
\begin{center}
if a word \(\overline{u}\) belongs to \(\mathcal{U}\), then so does every one of its subwords.
\end{center}
In this situation, the generalized Wronskian \(W_{\mathcal{U}}\) is called a \textsl{pure geometric generalized Wronskian}.
\end{definition}
In order to justify this terminology, we show that pure geometric generalized Wronskian are indeed geometric generalized Wronskians. To this end, one uses a Leibniz rule for the operators \(\partial_{\overline{u}}\), given in Lemma \ref{lemma: Leibniz}.
\begin{proposition}
\label{prop: ggW pgg W}
A pure geometric generalized Wronskian in the sense of Definition \ref{def: pgg W} is a geometric generalized Wronskian in the sense of Definition \ref{def: gg W}.
\end{proposition}
\begin{proof}
Let \( \mathcal{U} \) be a full set of size \(m\), and let \(g, f_{0}, \dotsc, f_{m} \) be elements of \(\mathcal{O}_{\C^{p}}\). One must show the following equality:
\[
W_{\mathcal{U}}(gf_{0}, \dotsc, gf_{m})
=
g^{m+1} W_{\mathcal{U}}(f_{0}, \dotsc, f_{m}).
\]
Denote, for \(\overline{u}\) a word belonging to \(\mathcal{U}\), \(L_{\overline{u}}\) (resp. \(L'_{\overline{u}}\)) the line
\[
L_{\overline{u}}
\bydef
(\partial_{\overline{u}}f_{0}, \dotsc, \partial_{\overline{u}}f_{m}) 
\
\left(\text{resp.}
\
L'_{\overline{u}}
\bydef
\left(\partial_{\overline{u}}(gf_{0}), \dotsc, \partial_{\overline{u}}(gf_{m})\right) \right).
\]
By convention, set \(L_{\emptyset}=(f_{0}, \dotsc, f_{m})\) (resp. \(L'_{\emptyset}=(gf_{0}, \dotsc, gf_{m})\)). Observe that those lines are the lines of the matrix defining \(W_{\mathcal{U}}(f_{0}, \dotsc, f_{m})\) (resp. \(W_{\mathcal{U}}(gf_{0}, \dotsc, gf_{m})\)). 

One proceeds by describing operations on the lines yielding the wanted result. For sake of notation, one keeps the same notation for the lines after having performed an operation on them: e.g. if one transforms \(L'_{\overline{u}}\), then \(L'_{\overline{u}}\) still denotes the transformed line.
Start by taking \(\overline{u}\) a word of length \(1\) in the full set \(\mathcal{U}\), and make the operation on the lines 
\[
L'_{\overline{u}}
\leftarrow
L'_{\overline{u}}
-
\partial_{\overline{u}}g \cdot L_{\emptyset},
\]
which transforms \(L'_{\overline{u}}\) into \(g\cdot L_{\overline{u}}\) by the classic Leibniz rule.
Do this for all the elements of size \(1\) in \(\mathcal{U}\).
Take now \(\overline{u}\) a word of length \(2\) in \(\mathcal{U}\), and make the following operation on the lines
\begin{equation*}
\begin{aligned}
L'_{\overline{u}}
\leftarrow
L'_{\overline{u}}
-
\sum\limits_{\substack{\overline{u_{1}} \cdot \overline{u_{2}}=\overline{u} \\ \overline{u_{1}} \neq \emptyset}}
\frac{1}{g}
C_{\overline{u_{1}}, \overline{u_{2}}}
\partial_{\overline{u_{1}}}g
\cdot
 gL_{\overline{u_{2}}},
 \end{aligned}
 \end{equation*}
 where \(C_{\overline{u_{1}}, \overline{u_{2}}}\) are constants defined in Lemma \ref{lemma: Leibniz}.
This is a well defined operation, keeping in mind the effect of the previous transformations and the fact that \(\mathcal{U}\) is a full set. This operation transforms the line \(L'_{\overline{u}}\) into the line \(g\cdot L_{\overline{u}}\) by the formula Lemma \ref{lemma: Leibniz}. One performs this for all words of length \(2\) in \(\mathcal{U}\).
One keeps doing these operations until all the words in the full set \(\mathcal{U}\) are exhausted. The key features are on the one hand formula Lemma \ref{lemma: Leibniz}, which gives the equality
\[
\partial_{\overline{u}}(gf)
-
g
\partial_{\overline{u}}f
=
\sum\limits_{\substack{\overline{u_{1}} \cdot \overline{u_{2}}=\overline{u} \\ \overline{u_{1}} \neq \emptyset}}
 C_{\overline{u_{1}}, \overline{u_{2}}} 
\partial_{\overline{u_{1}}}g \cdot \partial_{\overline{u_{2}}}f,
\]
and on the other hand the fact that in the sum on the right, the words \(\overline{v_{2}}\) appearing belong to \(\mathcal{U}\) (since they are subwords of \(\overline{u}\)) and are of length strictly less than \(\overline{u}\): the previous operations made on the lines allow then to implement the right operation to transform \(L'_{\overline{u}}\) into \(g\cdot L_{\overline{u}}\).
To conclude, the multi-linearity and alternating property of the determinant yields the equality
\[
W_{\mathcal{U}}(gf_{1}, \dotsc, gf_{m})
=
g^{m+1} W_{\mathcal{U}}(f_{0}, \dotsc, f_{m}),
\]
which is what we wanted to prove.
\end{proof}

As a corollary of the proof of Proposition \ref{prop: gg W germ}, one then has the following:
\begin{proposition}
\label{prop: Wronskian jet}
Let \(s_{0}, \dotsc, s_{m} \) be global sections of the line bundle \(L \). Let \( \mathcal{U} \) be a full set of size \(m\) and order \(k\), with caracteristic exponent \(\boldsymbol{\beta}=\boldsymbol{\beta}(\mathcal{U})\). Then \(W_{\mathcal{U}}(s_{0}, \dotsc, s_{m}) \) is in fact a global section of \(E_{p,k,\boldsymbol{\beta}}X \otimes L^{m+1}\).
\end{proposition}

Note that in the case where \(p=1\), one easily shows that \(W(s_{0}, \dotsc, s_{m})\) defines in fact a section of \(E_{1,k,\frac{m(m+1)}{2}}^{\inv}X\otimes L^{m+1}\): see \cite{brotbek2017hyperbolicity}. This is no longer true for \(p>1\), unless for very specific Wronskians.
\begin{example}
If \(\mathcal{U}\) is the full set of order \(k\) containing every word of length \(i\) for \(1 \leq i \leq k \), i.e. \(\mathcal{U}=\mathcal{W}_{p, \leq k}\), then one checks that \(W_{\mathcal{U}}(s_{0}, \dotsc, s_{m})\) defines an invariant section: see Lemma \ref{lemma: Wronskian det}.
\end{example}

\subsection{Pure geometric generalized Wronskians and linear dependance}
\label{sse: proof linear independance}
The goal of this section is to prove that, in the statement of Theorem \ref{thm: generalized W}, one can replace ``generalized Wronskians`` with ``pure geometric generalized Wronskians``:
\begin{theorem}
\label{thm: geometric generalized W}
Let \(f_{0}, \dotsc, f_{m} \) be \((m+1)\) be \(m+1\) holomorphic functions on \(\C^{p}\). They are linearly independant if and only if there exists a full set \( \mathcal{U} \) of size \( m \) such that \( W_{\mathcal{U}}(f_{0}, \dotsc, f_{m}) \not\equiv 0 \).
\end{theorem}
Following verbatim the scheme of proof of Theorem \ref{thm: generalized W} (as well as the notations), recall that it is enough to prove the following: if the \(p\)-uples  \( \left(\boldsymbol{\alpha_{i}}\right)_{0 \leq i \leq m} \) are distinct, then there exists a full set \(\mathcal{U}\) of size \(m\) such that 
\[
W_{\mathcal{U}}(\boldsymbol{z}^{\boldsymbol{\alpha}_{0}}, \dotsc, \boldsymbol{z}^{\boldsymbol{\alpha}_{m}})
\not\equiv
0.
\]

In the first part of this section, we introduce \textsl{geometric Vandermondes} as well as relevant algebraic sets associated to them. Provided a suitable description of these sets, we end the first subsection with the proof of Theorem \ref{thm: geometric generalized W}.
The second part of this section, longer and more technical, is devoted to the proof of the description of these algebraic sets.

\subsubsection{Geometric Vandermondes}
Let \(p \in \N_{\geq 1}\), and let \(m \in \N_{\geq 1}\). Denote \(F_{p,m}\) the set of full sets of size \(m\), whose elements are words in \(\mathcal{W}_{p}\).
For \(\overline{u}\) a word in \( \mathcal{W}_{p} \), denote
\[
X_{\overline{u}}(m)
\bydef
(\prod\limits_{k} x_{u_{k},0}, \dotsc, \prod\limits_{k} x_{u_{k},m}),
\]
where the product runs implicitely on the letters of the word \(\overline{u}=u_{1}u_{2} \dotsc \), and where 
\((x_{k,\ell})_{\substack{1 \leq k \leq p \\ 0 \leq \ell \leq m}}\) are indeterminates. 
By convention, one sets for the empty word:
\[
X_{\emptyset}(m)
=
(\underbrace{1, \dotsc, 1}_{\times (m+1)}).
\]
Using the \textsl{Kronecker product}, one has the equality:
\[
X_{\overline{u}}(m)
=
X_{1}^{\alpha_{1}(\overline{u})}(m) \dotsc X_{p}^{\alpha_{p}(\overline{u})}(m).
\]

Let \(\mathcal{U} \in F_{p,m} \). Denote \(M_{\mathcal{U}} \) (resp. \(\tilde{M}_{\mathcal{U}}\)) the following square matrix 
\[
M_{\mathcal{U}}
=
\begin{pmatrix}
X_{\emptyset}(m)
\\
X_{\overline{u_{1}}}(m)
\\
\cdot
\\
\cdot
\\
X_{\overline{u_{m}}}(m)
\end{pmatrix}
\ \
\left( \text{resp.}
\ \
\tilde{M}_{\mathcal{U}}
=
\begin{pmatrix}
X_{\overline{u_{1}}}(m-1)
\\
\cdot
\\
\cdot
\\
X_{\overline{u_{m}}}(m-1)
\end{pmatrix}
\right),
\]
where \(\Set{\overline{u_{1}}, \dotsc, \overline{u_{m}} } \) is the listing of the words in \( \mathcal{U} \) obtained by applying the algorithm of Remark \ref{rmk: order}.
\begin{definition}[Geometric Vandermondes]
The polynomial \(V_{\mathcal{U}} \bydef \det M_{\mathcal{U}} \) (resp. \(\tilde{V}_{\mathcal{U}} \bydef \det \tilde{M}_{\mathcal{U}}\)) is called a \textsl{geometric Vandermonde} (resp. an \textsl{homogeneous geometric Vandermonde}).
\end{definition}
Define the following algebraic sets:
\[
V_{p,m}
=
Z(V_{\mathcal{U}} \ | \ \mathcal{U} \in F_{p,m}) \subset (\C^{p})^{m+1}
\ \
; 
\ \ 
\tilde{V}_{p,m}
=
Z(\tilde{V}_{\mathcal{U}} \ | \ \mathcal{U} \in F_{p,m}) \subset (\C^{p})^{m}.
\]
Denote, for \( 0 \leq i \leq m \), 
\[
C_{i}(p)
=
\begin{pmatrix} x_{1,i} \\ \cdot \\ \\ \cdot \\ x_{p,i} \end{pmatrix}.
\]
For a full set \(\mathcal{U} \in F_{p,m}\), the polynomial \(V_{\mathcal{U}}\) (resp. \(\tilde{V}_{\mathcal{U}}\)) can be seen as depending on the variables \(\left(C_{1}(p), \dotsc, C_{m}(p)\right)\) (resp. \(\left(C_{1}(p), \dotsc, C_{m-1}(p)\right))\). As the index \(p\) is already specified in the nature of set \(\mathcal{U}\), we will write
\[
V_{\mathcal{U}}=V_{\mathcal{U}}(C_{1}, \dotsc, C_{m})
\ 
\left(\text{resp.}
\
\tilde{V}_{\mathcal{U}}=\tilde{V}_{\mathcal{U}}(C_{1}, \dotsc, C_{m-1})
\right).
\]
Note that this polynomial can also be seen as depending on the variables \(\left(X_{1}(m), \dotsc, X_{p}(m)\right)\) (resp. \(\left(X_{1}(m-1), \dotsc, X_{p}(m-1)\right)\)), so that we will also sometimes write
\[
V_{\mathcal{U}}=V_{\mathcal{U}}(X_{1}, \dotsc, X_{p})
\ 
\left(\text{resp.}
\
\tilde{V}_{\mathcal{U}}=\tilde{V}_{\mathcal{U}}(X_{1}, \dotsc, X_{p})
\right).
\]
(Once again we drop the index \(m\), as it is implicit).
We aim at proving the following description of the sets \(V_{p,m} \) and \(\tilde{V}_{p,m}\).
\begin{theorem}
\label{mainthm A}
Let $p,m \in \N_{\geq 1}$. The following holds:
\begin{center}
\((C_{0}, \dotsc, C_{m}) \in (\C^{p})^{m+1}\) belongs to \(V_{p,m}\) if and only if there exists two distinct indexes \(i \neq j\) such that \(C_{i}=C_{j}\).
\end{center}
\end{theorem}

\begin{theorem}
\label{mainthm B}
Let $p,m \in \N_{\geq 1}$. The following holds:
\begin{center}
\((C_{0}, \dotsc, C_{m-1}) \in (\C^{p})^{m}\) belongs to \(\tilde{V}_{p,m}\) if and only if there exists an index \(i\) such that \(C_{i}=0\), or two disctinct indexes \(i \neq j\) such that \(C_{i}=C_{j}\).
\end{center}
 \end{theorem}
 Before proceeding to the (simultaneous) proof of the two previous theorems, let us finish the proof of Theorem \ref{thm: geometric generalized W}.
 \begin{proof}[Proof of Theorem \ref{thm: geometric generalized W}]
Suppose that the \(p\)-uples  \( \left(\boldsymbol{\alpha_{i}}\right)_{0 \leq i \leq m} \) are distinct.  We must show that there exists a full set of size \(m\) such that 
\[
W_{\mathcal{U}}(\boldsymbol{z}^{\boldsymbol{\alpha}_{0}}, \dotsc, \boldsymbol{z}^{\boldsymbol{\alpha}_{m}})
\not\equiv
0.
\]
By Proposition \ref{prop: ggW pgg W}, the functional \(W_{\mathcal{U}}\) is a geometric generalized Wronskian. By multiplying each monomial \(\boldsymbol{z}^{\boldsymbol{\alpha}_{i}}\) by the same monomial \(\boldsymbol{z}^{\boldsymbol{\alpha}}\), one obtains therefore the equality
\[
W_{\mathcal{U}}(\boldsymbol{z}^{\boldsymbol{\alpha}+\boldsymbol{\alpha}_{0}}, \dotsc, \boldsymbol{z}^{\boldsymbol{\alpha}+\boldsymbol{\alpha}_{m}})
=
\boldsymbol{z}^{(m+1)\boldsymbol{\alpha}}
W_{\mathcal{U}}(\boldsymbol{z}^{\boldsymbol{\alpha}_{0}}, \dotsc, \boldsymbol{z}^{\boldsymbol{\alpha}_{m}}).
\]
Therefore, up to replacing the \(p\)-uples  \( \left(\boldsymbol{\alpha_{i}}\right)_{0 \leq i \leq m} \) by the \(p\)-uples
\( \left(\boldsymbol{\alpha}+\boldsymbol{\alpha_{i}}\right)_{0 \leq i \leq m} \), one can always suppose that the \(\boldsymbol{\alpha_{i}}\) are large enough so that for any word \(\overline{u}\) appearing in a full set \(\mathcal{U} \) of size \(m\), \(\partial_{\overline{u}} \boldsymbol{z}^{\boldsymbol{\alpha}_{i}}\) is not identically zero.

To conclude, the key observation is the following equality for any full set \(\mathcal{U}\) of size \(m\)
\[
W_{\mathcal{U}}(\boldsymbol{z}^{\boldsymbol{\alpha}_{0}}, \dotsc, \boldsymbol{z}^{\boldsymbol{\alpha}_{m}})(1)
=
V_{\mathcal{U}}(\boldsymbol{\alpha}_{0}, \dotsc, \boldsymbol{\alpha}_{m}),
\]
which is shown using appropriate operations on the lines (in a fashion very similar to the proof of Proposition \ref{prop: ggW pgg W}) . The result then follows immediately from Theorem \ref{mainthm A}.
 \end{proof}

 \subsubsection{Proof of Theorem \ref{mainthm A} and \ref{mainthm B}}

Let us denote \(\mathcal{H}_{p, m}\) (resp. \(\tilde{\mathcal{H}}_{p, m}\)) the property that Theorem \ref{mainthm A} (resp. Theorem \ref{mainthm B}) holds for the values \(p, m\). We first prove the following elementary lemma which relates these two properties, and follows from definitions:
\begin{lemma}
\label{lemma: rec}
Let \(p, m \in \N_{\geq 1}\). If \(\mathcal{H}_{p,m}\) holds, then so does \(\tilde{\mathcal{H}}_{p,m}\).
\end{lemma}
\begin{proof}
Let \((C_{0}, \dotsc, C_{m-1})\) be in \(\tilde{V}_{n, m}\), and denote \(0\) the zero column in \(\C^{p}\). Observe that
\[
M_{\mathcal{U}}(C_{0}, \dotsc, C_{m-1}, 0)
=
\tilde{M}_{\mathcal{U}}(C_{0}, \dotsc, C_{m-1}),
\]
so that \(V_{\mathcal{U}}(C_{0}, C_{1}, \dotsc, C_{m-1}, 0)=0\) for every \(\mathcal{U} \in F_{p,m}\). Since \(\mathcal{H}_{p,m}\) holds, this implies that two of these columns are equal, which gives exactly that either there exists \(0 \leq i \leq m-1\) such that \(C_{i}=0\) or \(1 \leq i < j \leq m-1\) such that \(C_{i}=C_{j}\).
\end{proof}

Accordingly, if Theorem \ref{mainthm A} holds, so does Theorem \ref{mainthm B}. Before diving into the proof Theorem \ref{mainthm A}, we start with the following lemma:
\begin{lemma}
\label{lemma: key lemma}
For any full set $\mathcal{U} \in F_{p,m}$, and any column 
\[
C=\begin{pmatrix} \lambda_{1} \\ \cdot \\ \cdot \\ \lambda_{p} \end{pmatrix},
\]
where the $\lambda_{i}$'s are indeterminates, the following equality holds:
\[
V_{\mathcal{U}}(C_{0}, \dotsc, C_{m})
=
V_{\mathcal{U}}(C_{0}+C, \dotsc, C_{m}+C).
\]
\end{lemma}
\begin{proof}
Consider the surjective map 
\[
\pi_{1}: \mathcal{U} \rightarrow \mathcal{U}'=\pi_{1}(\mathcal{U}),
\ \
\overline{u} \mapsto \overline{u}'=2^{\alpha_{1}(\overline{u})} \cdot \dotsb \cdot {p}^{\alpha_{p}(\overline{u})},
\]
whose fibers form a partition of \(\mathcal{U}\).
Pick an element \(\overline{u}'\) in  \(\mathcal{U'}\), and observe that since \(\mathcal{U}\) is a full set, there exists \(k_{1} \geq 0\) such that
\[
\pi_{1}^{-1}(\{\overline{u}'\})    
=
\{\overline{u}', 1 \cdot \overline{u}', \dotsc, 1^{k_{1}} \cdot \overline{u}' \}.
\]
Consider the lines \(X_{\overline{u}'}, X_{1\cdot \overline{u}'}, \dotsc, X_{1^{k_{1}} \cdot \overline{u}'}\). Denote \(Q=\prod\limits_{j=2}^{p} T_{j}^{\alpha_{j}(\overline{u}')} \in \Z[T_{2}, \dotsc, T_{n}]\), and observe that by definition
\[
X_{1^{i} \cdot \overline{u}'}
=
\big(x_{1,0}^{i}Q(x_{2,0}, \dotsc, x_{p,0}), \dotsc, x_{1,m}^{i}Q(x_{2,m}, \dotsc, x_{p,m})\big).
\]
Make then successively the following operations on the lines:
\begin{itemize}
\item{}
\(X_{1^{k_{1}} \cdot \overline{u}'} 
\leftarrow
X_{1^{k_{1}} \cdot \overline{u}'} 
+
\binom{k_{1}}{1} \lambda_{1} X_{1^{k_{1}-1}\cdot \overline{u}'} 
+
\dotsb
+
\lambda_{1}^{k_{1}}X_{\overline{u'}}\)
\item{}
\(X_{1^{k_{1}-1} \cdot \overline{u}'} 
\leftarrow
X_{1^{k_{1}-1} \cdot \overline{u}'} 
+
\binom{k_{1}-1}{1} \lambda_{1} X_{1^{k_{1}-2}\cdot \overline{u}'} 
+
\dotsb
+
\lambda_{1}^{k_{1}-1}X_{\overline{u'}}\)
\item{} \(\cdot\)
\item{} \(\cdot\)
\item{}
\(X_{1\cdot \overline{u}'} \leftarrow X_{1\cdot \overline{u}'}  + \lambda_{1} X_{\overline{u'}}\),
\end{itemize}
after which the line \(X_{1^{i} \cdot \overline{u}'}\) has been replaced by
\[
\big((x_{1,0}+\lambda_{1})^{i}Q(x_{2,0}, \dotsc, x_{p,0}), \dotsc, (x_{1,m}+\lambda_{1})^{i}Q(x_{2,m}, \dotsc, x_{p,m})\big).
\]

By doing this operation for every fiber, observe that one simply replaces, in the arguments of \(V_{\mathcal{U}}\), the line \(X_{1}\) by the line \(X_{1}+(\lambda_{1}, \dotsc, \lambda_{1})=(x_{1,1}+\lambda_{1}, \dotsc, x_{1,m} + \lambda_{1})\). Therefore, the following equality is proved:
\[
V_{\mathcal{U}}(X_{1}+(\lambda_{1}, \dotsc, \lambda_{1}), X_{2}, \dotsc, X_{p})
=
V_{\mathcal{U}}(X_{1}, X_{2}, \dotsc, X_{p}).
\]

By considering the other projections
\[
\pi_{j}:
\overline{u} \mapsto \overline{u}'
=
1^{\alpha_{1}(\overline{u})}  \dotsb  (j-1)^{\alpha_{j-1}(\overline{u})}  \cdot (j+1)^{\alpha_{j+1}(\overline{u})}  \dotsb {p}^{\alpha_{p}(\overline{u})}
\]
and doing the same reasoning, one finishes the proof of the lemma.
\end{proof}

We now turn to the proof of Theorem \ref{mainthm A} (giving also Theorem \ref{mainthm B} by Lemma \ref{lemma: rec}).

\begin{proof}[proof of theorem \ref{mainthm A}]
Proceed by induction on \((m+p)\geq 1\), where \(m \in \N_{\geq 0}\) and \(p \in \N_{\geq 1}\). One first proves the result for the extremal cases, i.e. for \(m=0\), \(p \in \N_{\geq 1}\) arbitrary, as well as for \(p=1\), \(m \in \N_{\geq 0}\) arbitrary.  If \(m=0\), there is nothing to prove since the condition is empty, and the algebraic set \(V_{p,1}\) is the whole set. If \(p=1\), the only full set in $\mathcal{W}_{1}$ of size \(m\) is the set $\mathcal{U}=\{1, 1^{2}, \dotsc, 1^{m}\}$, and the polynomial associated is
\[
V_{\mathcal{U}}
=
\begin{vmatrix}
1 & 1 & \cdot & \cdot & 1
\\
x_{1,0} & x_{1,1} & \cdot & \cdot & x_{1,m}
\\
\cdot & \cdot & \cdot & \cdot & \cdot
\\
\cdot & \cdot & \cdot & \cdot & \cdot
\\
x_{1,0}^{m} & x_{1,1}^{m} & \cdot & \cdot & x_{1,m}^{m}
\end{vmatrix},
\]
which is the usual Vandermonde determinant: the result is accordingly well-known.

Suppose the result granted for any \(m \in \N_{\geq 1}\), \(p \in \N_{\geq 2}\) with \(m+p \geq r \geq 3\). The goal is prove it for every \(m \in \N_{\geq 1}\), \(p  \in \N_{\geq 2}\) with \(m+p=(r+1)\). 

Let \((C_{0}, \dotsc, C_{m})\) be a zero of \(V_{p,m}\). By considering only the full sets \(\mathcal{U}\) belonging to \(F_{p-1,m} \subset F_{p,m}\), one deduces the following by induction hypothesis:
\[
\exists \ 0\leq i\neq j \leq m, \ \hat{C_{i}}=\hat{C_{j}},
\]
where given a column \(C\) of size \(p\), the column \(\hat{C}\) is the column of size \(p-1\) obtained by suppressing the last entry of \(C\).
Without loss of generality, suppose that \(i=0\) and \(j=1\).
By Lemma \ref{lemma: key lemma}, one obtains that \((0, C_{1}-C_{0}, \dotsc, C_{m}-C_{0})\) is also a zero of \(V_{p,m}\). Compute then that for every \(\mathcal{U} \in F_{p,m}\), the following equality holds:
\[
V_{\mathcal{U}}(0,C_{1}-C_{0}, \dotsc, C_{m}-C_{0})
=
\tilde{V}_{\mathcal{U}}(C_{1}-C_{0}, \dotsc, C_{m}-C_{0}).
\]
Accordingly, one is reduced to showing that \(\tilde{\mathcal{H}}_{n,m}\) holds, in a \textsl{particular case}: indeed, the column \(C_{1}-C_{0}\) is zero everywhere except possibly at the last line (since \(\hat{C_{1}}=\hat{C_{0}}\)). 
Consider accordingly the following algebraic set:
\[
X
\bydef
\Big\{(x_{p,1}, C_{2}, \dotsc, C_{m}) \subset \C \times (\C^{p})^{m-1} 
\ | \   
\big(\underbrace{\begin{pmatrix} 0 \\ \cdot \\ \cdot \\ 0 \\ x_{p,1} \end{pmatrix}}_{\bydef C_{1}}, C_{2}, \dotsc, C_{m}\big)
\in 
\tilde{V}_{p,m}
\Big\}.
\]
One must show that if \(\left(x_{p,1}, C_{2}, \dotsc, C_{m}\right) \in I\), then one of the following holds:
\[
\boldsymbol{(*)}
\
\
\
\
 \left(
 \exists \ 1 \leq i \leq m, 
 \
 C_{i}=0
 \right)
 \
 \text{or}
 \
 \left(
\exists \ 1\leq i < j \leq m,
\
C_{i}=C_{j}
\right).
\]
\( \)
\newline


Denote \(\pi: X \rightarrow (\C^{p})^{m-1}\) the projection onto the second factor, and
observe that \(\pi(X\cap\Set{x_{p,1}\neq 0})\) is included in the algebraic set
\[
Y
\bydef
\{
(C_{2}, \dotsc, C_{m}) 
\ | \ 
(\hat{C_{2}}, \dotsc, \hat{C_{m}}) \in \tilde{V}_{p-1,m-1}
 \}.
\]
Indeed, there is an inclusion 
\[
a\colon F_{p-1,m-1} \hookrightarrow F_{p,m}
\] 
obtained by sending \(\mathcal{U} \in F_{p-1,m-1}\) to \(a(\mathcal{U})\bydef \{p\} \cup \mathcal{U} \in F_{p,m}\). One computes that
\[
\tilde{V}_{a(\mathcal{U})}(\begin{pmatrix} 0 \\ \cdot \\ \cdot \\ 0 \\ x_{p,1} \end{pmatrix}, C_{2}, \dotsc, C_{m})
=
x_{p,1} \tilde{V}_{\mathcal{U}}(\hat{C_{2}}, \dotsc, \hat{C_{m}}),
\]
so that the inclusion follows. 

Denote for \(2 \leq i \leq m\)
\[
Y_{i}
=
\{
(C_{2}, \dotsc, C_{m}) \in Y \ | \ \hat{C_{i}}=0
\},
\]
as well as for \(2 \leq i < j \leq m\)
\[
Y_{i,j}
=
\{
(C_{2}, \dotsc, C_{m}) \in Y \ | \ \hat{C_{i}}=\hat{C_{j}}
\}.
\]
By induction hypothesis, the closed subsets \(Y_{i}\) and \(Y_{i,j}\) form the irreducible components of \(Y\).
Consider 
\[
O_{i}
=
Y_{i} \cap \{ x_{p,i} \neq 0\} \cap \{\hat{C_{k}} \neq \hat{C_{\ell}} \ \forall \ 2 \leq k < \ell \leq m \}
\]
as well as
\[
O_{i,j}
= 
Y_{i,j}
\cap
\{x_{p,i} \neq x_{p, j}\}
\cap
\{x_{1,i} \dotsb x_{p-1,i} \neq 0 \}
\cap
\{ \hat{C_{k}} \neq \hat{C_{\ell}} \  \forall \ 2 \leq k < \ell \leq m,  (k,\ell) \neq (i,j)
\}.
\]
Those are dense open subsets of \(Y_{i}\) and \(Y_{i,j}\) respectively.
\newline

One first shows that for \(2 \leq i \leq m\), the map
\[
\pi_{i}: \pi^{-1}(Y_{i}) \rightarrow Y_{i}
\]
is generically \(2-1\). More precisely, the following holds for any \((C_{2}, \dotsc, C_{m}) \in O_{i}\):
\begin{equation}
\label{eq: eq1fiber}
\begin{aligned}
\pi^{-1}\Set{(C_{2}, \dotsc, C_{m})}
= 
\Set{(0, C_{2}, \dotsc, C_{m})
(x_{p,i}, C_{2}, \dotsc, C_{m})
}.
\end{aligned}
\end{equation}
It is straightforward to check that the fiber contains at least these two elements, so that one must show that there are no others.
Consider the injection
\[
b\colon F_{p-1, m-2} \hookrightarrow F_{p,m}
\]
obtained by sending \(\mathcal{U} \in F_{p-1, m-2}\) to \(b(\mathcal{U})=\{p, p^{2}\} \cup \mathcal{U}\). Consider the matrix
\[
\tilde{M}_{b(\mathcal{U})}(\begin{pmatrix} 0 \\ \cdot \\ \cdot \\ 0 \\ x_{p,1} \end{pmatrix}, C_{2}, \dotsc, C_{i-1}, \begin{pmatrix} 0 \\ \cdot \\ \cdot \\ 0 \\ x_{p,i} \end{pmatrix}, C_{i+1}, \dotsc, C_{m}).
\]
By developping its determinant according to the two lines corresponding to the words \(p\) and \(p^{2}\), one finds, up to a sign, the following equality:
\begin{equation*}
\begin{aligned}
&
\tilde{V}_{b(\mathcal{U})}(\begin{pmatrix} 0 \\ \cdot \\ \cdot \\ 0 \\ x_{p,1} \end{pmatrix}, C_{2}, \dotsc, C_{i-1}, \begin{pmatrix} 0 \\ \cdot \\ \cdot \\ 0 \\ x_{p,i} \end{pmatrix}, C_{i+1}, \dotsc, C_{m})
\\
&
=
\begin{vmatrix}
x_{p,1} & x_{p,i}
\\
x_{p,1}^{2} & x_{p,i}^2
\end{vmatrix}
\tilde{V}_{\mathcal{U}}(\hat{C_{2}}, \dotsc, \hat{C_{i-1}}, \hat{C_{i+1}}, \dotsc, \hat{C_{m}}).
\end{aligned}
\end{equation*}
On the one hand, by definition of \(O_{i}\) and by induction hypothesis, there exists at least one full set \(\mathcal{U} \in F_{p-1,m-2}\) such that \(\tilde{V}_{\mathcal{U}}(\hat{C_{2}}, \dotsc, \hat{C_{i-1}}, \hat{C_{i+1}}, \dotsc, \hat{C_{m}}) \neq 0\). On the other hand, by definition of \(O_{i}\), one has \(x_{p,i} \neq 0\). The equality \eqref{eq: eq1fiber} follows immediately.
\newline 

Second, one shows that for \(2 \leq i < j \leq m\)
\[
\pi_{i,j}: \pi^{-1}(Y_{i,j}) \rightarrow Y_{i,j}
\]
is generically \(1-1\). More precisely, the following holds for any \((C_{2}, \dotsc, C_{m})\) in \(O_{i,j}\):
\begin{equation}
\label{eq: eq2fiber}
\begin{aligned}
\pi^{-1}\Set{(C_{2}, \dotsc, C_{m})}
=
\Set{(0, C_{2}, \dotsc, C_{m})}.
\end{aligned}
\end{equation}
Once again, it is straightforward to see that the fiber contains at least this element, so that one must show that it is the only one.
By definition of \(O_{i,j}\), and by induction hypothesis, one can find a full set \(\mathcal{U}\) in \(F_{p-1, m-2}\) such that
\[
\tilde{V}_{\mathcal{U}}(\hat{C_{2}}, \dotsc, \hat{C_{i-1}}, \hat{C_{i+1}}, \dotsc, \hat{C_{m}})
\neq
0.
\]
Let \(1 \leq k \leq (p-1)\) be a letter appearing in \(\mathcal{U}\), and consider the full set \(\mathcal{U}'\) in \(F_{p,m}$ defined as follows:
\[
\mathcal{U}'
=
\{p, k \cdot p\} \cup \mathcal{U}.
\]
Consider the matrix 
\[
\tilde{M}_{\mathcal{U}'}(\begin{pmatrix} 0 \\ \cdot \\ \cdot \\ 0 \\ x_{p,1} \end{pmatrix}, C_{2}, \dotsc, C_{m}),
\]
and apply the following operation on the columns
\[
\Co_{i} \leftarrow \Co_{i} - \Co_{j},
\]
where \(\Co_{\ell}\) is the \(\ell\)th column of the previous matrix. Since \(\hat{C_{i}}=\hat{C_{j}}\), all the elements of the new column are zero, except for the elements corresponding to the words \(p\) and \(k\cdot p\), which are respectively \(x_{p,i}-x_{p,j}\) and \(x_{k,i}(x_{p,i}-x_{p,j})\). Develop now the determinant of the new matrix (which is equal to the determinant of the previous one) according to the lines corresponding to the words \(p\) and \(k \cdot p\) to find that
\begin{equation*}
\begin{aligned}
&
\tilde{V}_{\mathcal{U}'}(\begin{pmatrix} 0 \\ \cdot \\ \cdot \\ 0 \\ x_{p,1} \end{pmatrix}, C_{2}, \dotsc, C_{m})
\\
&
=
\begin{vmatrix}
x_{p,1} & x_{p,i}-x_{p,j}
\\
0 & (x_{p,i}-x_{p,j})x_{k,i}
\end{vmatrix}
\tilde{V}_{\mathcal{U}}(\hat{C_{2}}, \dotsc, \hat{C_{i-1}}, \hat{C_{i+1}}, \dotsc, \hat{C_{m}}).
\end{aligned}
\end{equation*}
Since by definition of \(O_{i,j}\), \(x_{k,i}\neq 0\) and \(x_{p,i} \neq x_{p,j}\), the very choice of \(\mathcal{U}\) implies the equality \eqref{eq: eq2fiber}. 
\newline

Now, conclude the proof as follows. Let \(x=(x_{p,1}, C_{2}, \dotsc, C_{m})\in X\) and let \(y=\pi(x)=(C_{2}, \dotsc, C_{m})\). If \(x_{p,1}=0\), then \(\boldsymbol{(*)}\) holds: suppose therefore that \(x_{p,1} \neq 0\), so that \(y\) lies in \(Y\).
Suppose first that \(y\) belongs to \(Y_{i}\) for some \(2\leq i \leq m\), and that its fiber with respect to \(\pi\) is finite. If \(x_{p,i}=0\), then \(C_{i}=0\), so that \(\boldsymbol{(*)}\) holds. Suppose therefore that \(x_{p,i}\neq 0\). Accordingly, the fiber \(\pi^{-1}(\Set{y})\) must be equal to
\[
\{(0, C_{2}, \dotsc, C_{m}), (x_{p,i}, C_{2}, \dotsc, C_{m})\}
\]
as these two elements lies in it, and one knows that the cardinal of the fiber cannot exceed \(2\) (since \(\pi_{i}\) is generically \(2-1\)). Therefore, one necessarily has that \(x_{p,1}=x_{p,i}\), so that \(C_{1}=C_{i}\) (recall that \(\hat{C_{i}}=0\) by definition of \(Y_{i}\)): the property \(\boldsymbol{(*)}\) holds.

Second, suppose that \(y\) belongs to \(Y_{i,j}\) for some \(2 \leq i<j \leq m\), and that its fiber with respect to \(\pi\) is finite. Then its fiber must necessarily be equal to
\[
\{(0,C_{2}, \dotsc, C_{m})
\},
\]
as it contains at most one element (since \(\pi_{i,j}\) is generically \(1-1\)), and this element obviously lies in it. In this case, one deduce that \(C_{1}=0\), so that \(\boldsymbol{(*)}\) holds.

Finally, suppose that \(y\) has infinite fiber. 
Define an injection 
\[
c: F_{p,m-1} \rightarrow F_{p,m}
\] 
as follows: for \(\mathcal{U} \in F_{p,m-1}\), let \(0\leq k \leq m\) be the greatest integer such that \(p^{k}\) is in \(\mathcal{U}\); define then \(c(\mathcal{U})=\{p^{k+1}\} \cup \mathcal{U}\). Now, observe that by developping according to the first column the determinant of the matrix
\[
\tilde{M}_{c(\mathcal{U})}(\begin{pmatrix} 0 \\ \cdot \\ \cdot \\ 0 \\ x_{p,1} \end{pmatrix}, C_{2}, \dotsc, C_{m}),
\]
one obtains a polynomial in \(x_{p,1}\) whose leading coefficient is \(\tilde{V}_{\mathcal{U}}(C_{2}, \dotsc, C_{m})\). For \(y=(C_{2}, \dotsc, C_{m})\) such that the fiber \(\pi^{-1}(\Set{y})\) is not finite, the previous polynomial has therefore an infinite number of roots (i.e. all the elements in the fiber) so that it must be zero. It implies that for every \(\mathcal{U} \in F_{p,m-1}\), one has the following:
\[
\tilde{V}_{\mathcal{U}}(C_{2}, \dotsc, C_{m})
=
0.
\]
By induction hypothesis, there exists either \(2 \leq i \leq m\) such that \(C_{i}=0\), or \(2 \leq i<j \leq m\) such that \(C_{i}=C_{j}\), and thus \(\boldsymbol{(*)}\) holds. This finishes the proof.
\end{proof}

\subsection{Filtration of geometric generalized Wronskians \& representation theory}
\label{sse: gg W and repr}
Recall that for any \(m \in \N_{\geq 1}\), the group of biholomorphisms of \((\C^{p},0)\), and in particular the linear group \(GL_{p}(\C)\), acts on the finite-dimensional \(\C\)-vector space \(\mathcal{W}_{g}(m)\). In this section, we establish a filtration of this representation, and we discuss an elementary property in representation theory: those will be needed for our first application in foliation theory (see Section~\ref{sse: foliation integrability}).
\begin{definition}[Caracteristic sequence]
Let \(\mathcal{U}\) be an admissible set of words of size \(m\), order \(k\) and weight \(w\). Define the \textsl{caracteristic sequence of \(\mathcal{U}\)} as the \(k\)-uple of positive integers
\[
\mathbi{n}(\mathcal{U})\bydef (n_{1}, \dotsc, n_{k}),
\]
where \(n_{i}\) is the number of words of length \(i\) belonging to \(\mathcal{U}\). Note that \(m=n_{1}+n_{2}+ \dotsb + n_{k} \) and \(w(\mathcal{U})=n_{1}+2n_{2}+\dotsb kn_{k}\).

Let \(m \in \N_{\geq 1}\), and let \(\mathbi{n}=(n_{1}, \dotsc, n_{k}) \in \N_{\geq 1}^{k}\) with \(k \in \N_{\geq 1}\).
The \(k\)-uple of positive integers \(\mathbi{n}\) is called a \textsl{caracteristic sequence of order \(k\) of the integer \(m\) } if and only if it is the caracteristic sequence of an admissible set (of size \(m\) and order \(k\)). The \textsl{weight} of the caracteristic sequence \(\mathbi{n}\) is by definition
\[
w(\mathbi{n})=n_{1}+2n_{2}+\dotsb kn_{k},
\]
i.e. it is the weight of any of the possible admissible sets associated to the caracteristic sequence \(\mathbi{n}\).
\end{definition}
One orders such caracteristic sequences \(\mathbi{n}=(n_{1}, \dotsc, n_{k})\)
\begin{itemize}
\item{} by increasing order of their order \(k\);
\item{} at fixed order \(k\), by \textsl{decreasing} order in the lexicographic order of \(\mathbb{N}^{k}\).
\end{itemize}
Namely, \(\mathbf{n'} < \mathbi{n} \) if and only if, either the order of \(\mathbf{n'}\) is strictly less than the order of \(\mathbi{n}\), or, if their order is equal to \(k\), if and only if \(\mathbf{n'}\) is greater than \(\mathbi{n}\) in the lexicographic  order of \(\mathbb{N}^{k}\). This defines a total order on caracteristic sequences.

Fix \(\mathbi{n}=(n_{1}, \dotsc, n_{k})\) a caracteristic sequence of \(m\), and define
\[
\mathcal{W}_{g}^{\leq \mathbi{n}}(m)
\bydef
\Set{
W \in \mathcal{W}_{g}(m)
\
|
\
W=\sum\limits_{\substack{\mathcal{U} \ \text{admissible}, \\ \mathbi{n}(\mathcal{U})\leq\mathbi{n}}} \lambda_{\mathcal{U}} W_{\mathcal{U}}, \ \lambda_{\mathcal{U}} \in \C
}.
\]
Observe that \(\mathcal{W}_{g}^{\leq \mathbi{n}}(m)\) is stable under the action of the biholomorphisms of \((\C^{p},0)\) (thus in particular under the action of the linear group \(GL_{p}(\C)\)), which follows immediately from the formula Lemma \ref{lemma: compo} to compute multi-derivatives of compositions of maps. This gives a filtration of the representation \(\mathcal{W}_{g}(m)\) by sub-representations of \(\mathcal{G}_{p,m}\). 
We will also denote
\[
\mathcal{W}_{g}^{<\mathbi{n}}(m)
\bydef
\Set{
W \in \mathcal{W}_{g}(m)
\
|
\
W=\sum\limits_{\substack{\mathcal{U} \ \text{admissible}, \\ \mathbi{n}(\mathcal{U})<\mathbi{n}}} \lambda_{\mathcal{U}} W_{\mathcal{U}}, \ \lambda_{\mathcal{U}} \in \C
}.
\]
\begin{remark}
The previous filtration  is also valid if one replaces \(\mathcal{W}_{g}(m)\) by \(\mathcal{W}(m)\).
\end{remark}

Let \(\mathcal{U}\) be an admissible set of size \(m\) and caracteristic exponent \(\mathbi{\beta}\). The generalized Wronskian \(W_{\mathcal{U}}\) is a weight vector of weight \(\mathbi{\beta}\) (in the usual sense in representation theory), i.e. one has the following equality
\begin{equation}
\label{eq: represent}
\Diag(\lambda_{1}, \dotsc, \lambda_{p}) \cdot W_{\mathcal{U}}
=
\lambda_{1}^{\beta_{1}} \dotsb \lambda_{p}^{\beta_{p}} W_{\mathcal{U}}
\end{equation}
for any \(\lambda_{1}, \dotsc, \lambda_{p}\) in \(\C^{*}\).

Suppose that \(\mathcal{V}\) is an irreducible representation of \(GL_{p}(\C)\) appearing in \(\mathcal{W}_{g}^{\leq \mathbi{n}}(m)\), and suppose that this representation does not come from the ones of \(\mathcal{W}_{g}^{< \mathbi{n}}(m)\). One knows that for any \(\lambda \in \C^{*}\) and any \(W \in \mathcal{V}\), one has the egality
\[
\Diag(\lambda, \dotsc, \lambda) \cdot W
=
\lambda^{r} W
\]
for some integer \(r \in \N_{\geq 1}\). In particular, this implies that every element of \(\mathcal{V}\) is an unmixed geometric generalized Wronskians of weight \(r\). Since a pure generalized Wronskian \(W_{\mathcal{U}}\), with \(\mathcal{U}\) an admissible set of caracteristic sequence \(\mathbi{n}(\mathcal{U})=\mathbi{n}\), must appear in the writing of at least one element \(W \in \mathcal{V}\) (recall that we assumed that \(\mathcal{V} \not\subset \mathcal{W}_{g}^{< \mathbi{n}}(m)\)), one deduces from \eqref{eq: represent} that necessarily
\[
r=w(\mathbi{n})=n_{1}+2n_{2}+ \dotsb=\beta_{1}+\beta_{2} + \dotsb.
\]
Then, by standard facts in representation theory, one deduces that there exists a partition \(\lambda\) of \(w(\mathbi{n})\) such that
\[
\mathcal{V} 
\simeq 
S^{\lambda} \C^{p}.
\]
This observation will be used in the course of the proof of Theorem \ref{thm: foliation}.
\section{Application in intermediate hyperbolicity}

\label{sse: hyperbolicity}
In this section, we start by recalling a fundamental vanishing theorem of Siu, which asserts that on a projective manifold \(X\), every entire curve must satisfy the differential equation defined by a global jet differentials vanishing along an ample. We then state a similar statement concerning non-degenerate holomorphic maps \(\C^{p}\to X\): every such map must satisfy the partial differential equation defined by a global partial jet differentials vanishing along an ample. We give the proof of this statement in Appendix \ref{appendix: vanishing}, following the approach of the original proof, but using Nevanlinna theory in several variables (instead of one variable).

Then, as an application of this vanishing theorem and our construction of global partial jet differentials via Wronskians, we prove Theorem \ref{thm: mainthm 2}, i.e. we show the algebraic degeneracy of non-degenerate holomorphic maps from \(\C^{p}\) to Fermat hypersurfaces of degree \(\delta\) in projective spaces \(\P^{N}\), as soon as \(\delta > (N+1)(N-p)\).

\subsection{A vanishing theorem}
Recall the following fundamental vanishing theorem, due to Siu:
\begin{theorem}[Siu]
\label{thm: vanishing Siu}
Let \(X\) be a projective manifold equipped with an ample line bundle \(L\), and suppose that there exists a (non-zero) global section 
\[
P \in H^{0}(X, E_{1,k,w}X\otimes L^{-1}) 
\]
where \(k,w \in \N_{\geq 1}\). Then any entire curve \(f: \C \rightarrow X\) satisfies the differential equation
\[
P(f)=P_{f}(\partial_{1,k}f)\equiv0.
\]
\end{theorem}
Recall that the conclusion of this theorem means that for any \(z \in \C\), the evaluation of \(P_{f(z)}\) on the \(1\)-germ 
\[
f(z+\cdot)\colon (\C,0) \to (X, f(z))
\] 
is zero.
This theorem was first proved by Siu using methods from Nevanlinna theory: see \cite{Dem} for a short exposition of the proof. Later, Demailly proved this vanishing result first for invariant
jet differentials bundles \(E_{1,k,w}^{\inv}\), using the Demailly-Semple tower and the Ahlfors-Schwarz lemma. Then, with a clever argument, he obtained the general statement for every jet differentials bundles (see  \cite{DemSurvey}[Chapter 7])).
In the setting of \(p\)-germs, Siu's scheme of proof can be carried over verbatim, with some Nevanlinna theoretic considerations in several variables. In Appendix \ref{appendix: vanishing}, we prove the following generalization of Theorem \ref{thm: vanishing Siu}:
\begin{theorem}
\label{thm: vanishing}
Let \(X\) be a projective manifold equipped with an ample line bundle \(L\), and suppose that there exists a (non-zero) global section
\[
P \in H^{0}(X, E_{p,k,w}X\otimes L^{-1}),
\]
where \(p,k,w \in \N_{\geq 1}\). Then any non-degenerate map \(f: \C^{p} \rightarrow X\) satisfies the partial differential equation
\[
P(f)=P_{f}(\partial_{p,k}f)\equiv0.
\]
\end{theorem}
\begin{remark}
The Demailly-Semple tower and the Ahlfors Schwarz lemma can be generalized in the setting of $p$-germs, \(p>1\), and Demailly's scheme of proof leads to a vanishing result (see \cite{Rousseau}). However, for \(p>1\), the bridge between the generalized Demailly-Semple tower and the \textsl{invariant} \(p\)-jet differentials is not completely clear (see Problem 2.9 in \cite{Rousseau}). Furthermore, while the gap between invariant and non-invariant jet bundles is small for \(p=1\) (see \cite{DemSurvey}), this is no longer true for \(p>1\).
\end{remark}

\subsection{Holomorphic maps from $\C^{p}$ to a Fermat hypersurface}
Recall the following standard definition:
\begin{definition}[Fermat hypersuface]
A Fermat hypersurface \(F\) of degree \(\delta \geq 1\) inside the projective space \(\P^{N}\), \(N \geq 1\), is any hypersurface defined by an equation of the following form
\[
\lambda_{0}X_{0}^{\delta} + \dotsb + \lambda_{N} X_{N}^{\delta}=0
\]
where \((\lambda_{0}, \dotsc, \lambda_{N}) \) is a non-zero vector in  \(\C^{N+1}\). 
\end{definition}

The goal of this section is to prove the following theorem:
\begin{theorem}
\label{thm: hyperbolicity}
Let \(H \subset \P^{N}\) be a Fermat hypersurface of degree \(\delta\). If
\[
\delta > (N+1)(N-p),
\]
then any non-degenerate holomorphic map
\(f\colon \C^{p} \rightarrow H\)
is algebraically degenerate. More precisely, there exists a Fermat hypersurface  \(H' \subsetneq \P^{N-1} \subsetneq \P^{N}\) of degree \(\delta\)  such that
\[
f(\C^{p}) \subset H' \cap H.
\]
\end{theorem}

Note that this interpolates between the following two well known results. On the one hand, if \(\delta > (N+1)\), \(H\) has  canonical divisor $K_{H}$ ample by the adjunction formula. Accordingly, the Fermat hypersurface \(H\) is of general type, which in turn implies that \(H\) is measure hyperbolic.  This property is enough to ensure the non-existence of non-degenerate families of entire curves \( \C \times \B^{N-2} \rightarrow H\), which implies in particular the non-existence of non-degenerate holomorphic maps \(\C^{N-1} \rightarrow H\). On the other hand, if \(\delta > (N+1)(N-1)
\), it is known that every holomorphic map \(\C \to H\) is algebraically degenerate: see e.g. \cite{kob13}[p.144-145] or \cite{Santa}[Chapter 11].

\begin{proof}
Suppose without loss of generality that \(F=\Set{X_{0}^{\delta}+ \dotsc + X_{N}^{\delta}=0}\), and that the image of \(f\) is not included in any of the coordinate hyperplane \(\Set{X_{i}=0}\).
Fix \(z_{0} \in \C^{p}\) a non-degenerate point of \(f\) and, up to translation, suppose that \(z_{0}=0\). Suppose furthermore that, for instance, \(X_{0}(f(0)) \neq 0\), so that one can write locally around \(0=z_{0}\):
\[
f
=
(f_{1}, \dotsc, f_{N}).
\]
Since \(f\) is non-degenerate at \(0\), there exists a biholomorphism \(\varphi: U_{1} \to U_{2}\) between neighborhoods of \(0\) such that, up to reordering the variables in $\P^{N}$, one has the following equality around \(0\)
\[
f\circ\varphi
=
(z_{1}, \dotsc, z_{p}, g_{p+1}, \dotsc, g_{N}),
\]
where the functions \(g_{i}\) are holomorphic and defined around \(0\). 

Consider the pure geometric generalized Wronskians associated to the \(N\) global sections \(X_{1}^{\delta}, \dotsc, X_{N}^{\delta}\) of \(\O_{\P^{N}}(\delta)\).
Partition \(F_{p,N-1}\) (which we recall is the set of full sets of size \(N-1\) with words in \(\mathcal{W}_{p}\)) into \(F_{+}\) and \(F_{-}\), where \(F_{+}\) is the set of full sets containing the \(p\) words \(\{1\}, \dotsc, \{p\}\), whereas \(F_{-}\) is its complementary. 

One first shows that, in a neighborhood of \(0=z_{0}\), one has the following equality for any \(\mathcal{U} \in F_{-}\):
\[
W_{\mathcal{U}}(X_{1}^{\delta}, \dotsc, X_{N}^{\delta})_{| H}(f\circ \varphi) \equiv 0.
\]
Indeed, using the equality 
 \(f_{1}^{\delta} + \dotsc f_{N}^{\delta}=-1\), valid around \(0\),
one can write (up to sign):
\[
W_{\mathcal{U}}(X_{1}^{\delta}, \dotsc, X_{N}^{\delta})_{| H}(f \circ \varphi)
\overset{loc}{=}
W_{\mathcal{U}}(z_{1}^{\delta}, \dotsc, z_{p}^{\delta}, g_{p+1}^{\delta}, \dotsc, g_{N-1}^{\delta}, -1).
\]
Since \(\mathcal{U}\) is in \(F_{-}\), there exists a letter \(1 \leq i \leq p\) such that no word in \(\mathcal{U}\) contains \(i\). Therefore, the \(i\)th column in the matrix defining 
\(
W_{\mathcal{U}}(z_{1}^{\delta}, \dotsc, z_{p}^{\delta}, g_{p+1}^{\delta}, \dotsc, g_{N-1}^{\delta}, -1)
\)
is equal to
\(
\begin{pmatrix}
z_{i}^{\delta}
\\
0 
\\
\cdot
\\
\cdot
\\
0 
\end{pmatrix},
\)
and is thus proportional to the last column
\(
\begin{pmatrix}
-1 
\\
0 
\\
\cdot
\\
\cdot
\\
0 
\end{pmatrix},
\)
which proves the announced equality, by alternating property of the determinant.

Next, one shows that the following equality is satisfied for any \(\mathcal{U} \in F_{+}\) and any \(\psi\) biholomorphism of \((\C^{p},0)\):
\[
(\psi \cdot W_{\mathcal{U}})(X_{1}^{\delta}, \dotsc, X_{N}^{\delta})_{| H}(f) \equiv 0.
\]
The idea is to use a factorization trick as in \cite{DemSurvey}, and apply the vanishing Theorem \ref{thm: vanishing}.
Observe that, by construction, any full set \(\mathcal{U} \in F_{+}\) has order  less or equal than \((N-p)\). Accordingly, by multi-linearity of the determinant, one can factorize \((\psi \cdot W_{\mathcal{U}})(X_{1}^{\delta}, \dotsc, X_{N}^{\delta})\) by 
\[
(X_{1} \dotsb X_{N})^{\delta- (N-p)}.
\]
By restricting \((\psi \cdot W_{\mathcal{U}})(X_{1}^{\delta}, \dotsc, X_{N}^{\delta})\) to the Fermat hypersurface \(H=\Set{X_{0}^{\delta}+ \dotsb + X_{N}^{\delta}=0}\), the multi-linearity and alternating property of the determinant gives the equality:
\[
(\psi \cdot W_{\mathcal{U}})(X_{1}^{\delta}, \dotsc, X_{N}^{\delta})_{| H}
=
(\psi \cdot W_{\mathcal{U}})(X_{1}^{\delta}, \dotsc, X_{N-1}^{\delta}, -X_{0}^{\delta})_{| H}.
\]
Therefore, one deduces that \((\psi \cdot W_{\mathcal{U}})(X_{1}^{\delta}, \dotsc, X_{N}^{\delta})_{| H}\) can in fact be factorized by \((X_{0}\dotsc X_{N})^{\delta - (N-p)}_{| H}\), so that one can write
\[
(\psi \cdot W_{\mathcal{U}})(X_{1}^{\delta}, \dotsc, X_{N}^{\delta})_{| H}
=
(X_{0}\dotsc X_{N})^{\delta - (N-p)}_{| H}
\cdot
\tilde{W_{\mathcal{U}}},
\]
where 
\(\tilde{W_{\mathcal{U}}} \in H^{0}\big(H, E_{p,k, \boldsymbol{\beta}}H \otimes L^{(N+1)(N-p)-\delta}\big)\), with \(k\) the order of \(\mathcal{U}\) and \(\boldsymbol{\beta}\) its caracteristic exponent (see Definitions \ref{def: order}).
Since \(\delta > (N+1)(N-p)\),  the vanishing Theorem \ref{thm: vanishing} yields the equality
\[
(\psi \cdot W_{\mathcal{U}})(X_{1}^{\delta}, \dotsc, X_{N}^{\delta})_{| H}(f) \equiv 0.
\]
Now, observe that this implies that in a neighborhood of \(0=z_{0}\), one has the following equality for any \(\mathcal{U} \in F_{+}\): 
\[
W_{\mathcal{U}}(X_{1}^{\delta}, \dotsc, X_{N}^{\delta})_{| H}(f \circ \varphi) \equiv 0.
\]
Indeed, for \(z\) in a neighborhood of \(0\), consider the biholomorphism \(\psi\) of \((\C^{p},0)\) induced by \(\varphi\)
\[
\psi\bydef \varphi(z+\cdot) - \varphi(z),
\]
and observe that the evaluation of \(W_{\mathcal{U}}(X_{1}^{\delta}, \dotsc, X_{N}^{\delta})_{| H}\) at the \(p\)-germ \((f\circ \varphi) (z+ \cdot)\) is equal to the evaluation of \((\psi \cdot W_{\mathcal{U}})(X_{1}^{\delta}, \dotsc, X_{N}^{\delta})_{| H}\) at the \(p\)-germ \(f(\varphi(z)+\cdot)\). 

All in all, one has shown that in a neighborhood of \(0=z_{0}\), one has the following equality for any \(\mathcal{U}\) full set of size \(m\)
\[
W_{\mathcal{U}}(X_{1}^{\delta}, \dotsc, X_{N}^{\delta})_{| H}(f \circ \varphi) \equiv 0.
\]
By Theorem \ref{thm: geometric generalized W}, one deduces the existence of constants \(\lambda_{1}, \dotsc, \lambda_{N}\) in \(\C\) such that, locally around \(0=z_{0}\), the following equality is satisfied:
\[
\sum\limits_{i=1}^{N} 
\lambda_{i} X_{i}^{\delta} \circ f 
=
0.
\]
This expression is analytic on \(\C^{p}\), so that the previous equality is valid everywhere. This shows that the image of \(f\) is included in the Fermat hypersurface
\[
\Set{\lambda_{1}X_{1}^{\delta}+ \dotsb \lambda_{N}X_{N}^{\delta}=0} \subset \P^{N-1} \subset \P^{N},
\]
which finishes the proof.
\end{proof}
Note that if one wishes to apply successively the previous theorem, one has to check that the induced map
\( \C^{p} \rightarrow F' \cap F \) remains non-degenerate. In the case where \(p=1\), this is always the case, which leads to the neat description given in \cite{kob13}.

\section{Applications in foliation theory}
\label{se: foliation intro}
Geometric generalized Wronskians must be seen as an effective way of producing sections of some specific kinds: for instance, we used them in the previous section to produce partial differential equations vanishing along an ample. In this section, we will only consider specific geometric generalized Wronskians: therefore, before stating our two applications in foliation theory, we set some notations and state a few elementary facts concerning the combinatorics of the full sets associated to these specific Wronskians.

For what follows, we fix \(p\) an integer. For any \(n \in \N_{\geq 1}\), we denote by \(\mathcal{U}_{n}\) the set \(\mathcal{W}_{p, \leq n}\): this is obviously a full set of words. The size of this set is equal to the dimension of the vector space spanned by the polynomials of degrees ranging from \(1\) to \(n\) in \(\C[X_{1}, \dotsc, X_{p}]\) , i.e.
\[
\abs{\mathcal{U}_{n}}
=
 \sum\limits_{i=1}^{n}
 \binom{p+i-1}{p-1}.
\]
From the above equality, we deduce immediately that the weight of the full set \(\mathcal{U}_{n}\) is then equal to
\[
w(\mathcal{U}_{n})
=
\sum\limits_{i=1}^{n} i \binom{p+i-1}{p-1}.
\]

We will especially be interested in the asymptotic behavior of \(\abs{\mathcal{U}_{n}}\) and \(w(\mathcal{U}_{n})\) for \(n \to \infty\). The above formula allows to obtain the following estimates:
\[
\abs{\mathcal{U}_{n}}
=
\frac{1}{(p-1)!}
\sum\limits_{i=1}^{n}\left(
i^{p-1}
+
O( i^{p-2}) \right)
\sim
\frac{1}{p!}n^{p}
\]
and
\[
w(\mathcal{U}_{n})
=
\frac{1}{(p-1)!}
\sum\limits_{i=1}^{n}\left(
i^{p}
+
O( i^{p-1}) \right)
\sim \frac{p}{(p+1)!} n^{p+1}.
\]

Let us now briefly recall basic facts concerning foliations on \textsl{normal} varieties.
Recall first that, in the setting of normal  varieties, one defines the \textsl{tangent sheaf \(TX\)} of \(X\) as the dual of the sheaf \(\Omega_{X}\) of Khäler differentials on \(X\), i.e. \(TX \bydef \Omega_{X}^{*}\). A foliation on a normal variety is then defined as follows:
\begin{definition}[Foliations]
\label{def: foliation}
A foliation \(\mathcal{F}\) on a normal variety \(X\) is a coherent subsheaf \(T_{\mathcal{F}}\) of \(TX\) such that
\begin{enumerate}[(i)]
\item{} \(T_{\mathcal{F}}\) is closed under Lie bracket;
\item{} \(T_{\mathcal{F}}\) is saturated in \(TX\) i.e. the quotient \(TX/T_{\mathcal{F}}\) is torsion-free.
\end{enumerate}
\end{definition}
What will be important for us is the following fact: outside a singular set \(\Sing(\mathcal{F}) \supset \Sing(X)\) of codimension at least \(2\), the foliation \(\mathcal{F}\) is \textsl{regular}. This means that, in a neighborhood \(U_{x}\) of any point \(x \in X \setminus \Sing(\mathcal{F})\), the subsheaf \(T\mathcal{F}_{\vert U_{x}}\) is a subbundle of the holomorphic tangent sheaf \(TX_{\vert U_{x}}\), which can be integrated (via the usual Frobenius theorem, thanks to condition \((i)\) in the definition of a foliation). This leads to the following definition of \textsl{leaves of a foliation}:
\begin{definition}[Leaves of a foliation, algebraic leaves]
Let \(X\) be a normal variety, and let \(\mathcal{F}\) be a foliation on \(X\). A \textsl{leaf of \(\mathcal{F}\)} is a connected, locally closed holomorphic submanifold \(F \subset X \setminus \Sing(\mathcal{F})\) such that the tangent bundle \(TF\) of \(F\) is equal to \((T_{\mathcal{F}})_{\vert F}\).
A leaf \(F\) is called \textsl{algebraic} if it is open in its Zariski closure.
\end{definition}

Recall that any \textsl{reflexive} sheaf \(\mathcal{E}\) (i.e. a coherent sheaf isomorphic to its double dual \(\mathcal{E}^{**}\)) on a normal variety \(X\) satisfies Serre's \(S_{2}\) condition (see \cite{Har94}[Theorem 1.9]). Therefore, since \(\codim(\Sing(\mathcal{F}))\geq 2\),  one has the following isomorphism
\[
H^{0}(X \setminus \Sing(\mathcal{F}),\mathcal{E})
\simeq
H^{0}(X, \mathcal{E}),
\]
i.e. every section of \(\mathcal{E}\) defined on \(X \setminus \Sing(\mathcal{F})\) extends uniquely to a global section of \(\mathcal{E}\) (see \cite{Har94}[Prop 1.11]).

As outside \(\Sing(\mathcal{F})\) the local picture of the foliation is explicit, the aforementioned result leads to consider especially reflexive sheaves. This motivates the following usual definition:
\begin{definition}[Reflexive hull, determinant]
\label{def: reflexive}
Let \(X\) be a normal projective variety, and let \(\mathcal{E}\) a coherent sheaf on \(X\). For any \(m \in \N\), define the \textsl{reflexive hull of \(\mathcal{E}^{\otimes m}\)}, denoted \(\mathcal{E}^{[m]}\), as the following double dual:
\[
\mathcal{E}^{[m]}
\bydef
\left(\mathcal{E}^{\otimes m}\right)^{**}.
\]
The \textsl{determinant of \(\mathcal{E}\)}, denoted \(\det(\mathcal{E})\), is defined similarly
\[
\det(\mathcal{E})
\bydef
\left(\Ext^{\rank(\mathcal{E})} \mathcal{E}\right)^{**}.
\]
\end{definition}
Note that, when \(X\) is regular (e.g. \(X\) smooth), the determinant of a coherent sheaf defines an invertible sheaf (see \cite{Har80}[Prop. 1.9]), and thus a line bundle. In any event, there is a canonical reflexive sheaf of rank one associated to any foliation on a normal variety, called the \textsl{canonical bundle of the foliation}:
\begin{definition}[Cotangent bundle, canonical bundle of a foliation]
Let \(X\) be a normal variety, and let \(\mathcal{F}\) be a foliation on \(X\). The \textsl{canonical bundle of \(\mathcal{F}\)}, denoted \(\mathcal{K}_{\mathcal{F}}\), is defined as follows
\[
K_{\mathcal{F}}
\bydef
\det(\Omega_{F}),
\]
where \(\Omega_{F} \bydef T_{\mathcal{F}}^{*}\) is the \textsl{cotangent bundle of the foliation}.
\end{definition}

\subsection{Algebraicity of leaves}
\label{sse: foliation integrability}
Using the construction of Wronskians in the case where \(p=1\), one has the following application regarding foliations by curves on normal projective varieties. This was pointed out to me by Erwan Rousseau, following an observation of Jorge Vitorio Pereira.
\begin{proposition}
\label{prop: fol rk 1}
Let \(X\) be a normal projective variety and let \(\mathcal{F}\) be a foliation by curves on \(X\). 
Suppose that \(\mathcal{F}\) is not algebraically integrable, i.e. that a general leaf of \(\mathcal{F}\) is not algebraic. Then the canonical bundle of the foliation \(K_{\mathcal{F}}\) is pseudo-effective.
\end{proposition}
\begin{proof}
Polarize the projective variety \(X\) with an ample line bundle \(L\), and denote \(N=\dim(X)\). Since the general leaf of \(\mathcal{F}\) is not algebraic, there exists a constant \(C \in \N_{\geq 1}\) such that for any \(r \in \N_{\geq 1}\), there is at least \(Cr^{2}\) sections in \(H^{0}(X, L^{r})\) that are linearly independant when restricted to a general leaf \(F\):
let \(s_{0}, \dotsc, s_{m}\) be \(m+1=Cr^{2}\) such sections.

For what follows, one works on the open set \(X \setminus \Sing(\mathcal{F})\).
Denote \(\gamma\colon  (\C^{N-1},0) \times (\C,0) \to X \) a local integration of the foliation around a point \(x \in X \setminus \Sing(\mathcal{F})\). Therefore, the leaves of the foliation are (locally) parametrized by \(\gamma_{\mathbi{t}}\colon (\C,0) \to X\), where \(\mathbi{t}\) denotes coordinates on \((\C^{N-1},0)\). Two local integrations of the foliation differ by a family of biholomorphisms \(\varphi_{\mathbi{t}}\colon (\C,0) \to (\C,0)\). Evaluating the Wronskian \(W(s_{0}, \dotsc, s_{m})\) at the germs \(\gamma_{\mathbi{t}}(z+\cdot)\) and \(\gamma_{\mathbi{t}} \circ \varphi_{\mathbi{t}}(z+\cdot)\) for \(z\) in a neighborhood of \(0\) yields two numbers that are related as follows:
\[
W(s_{0}, \dotsc, s_{m})(\gamma_{\mathbi{t}}\circ \varphi_{\mathbi{t}}(z + \cdot))
=
(\varphi_{\mathbi{t}}'(z))^{\frac{m(m+1)}{2}} W(s_{0}, \dotsc, s_{m})(\gamma_{\mathbi{t}}(z+ \cdot)).
\]
Furthermore, they are non identically zero since the sections are linearly independant when restricted to a general leaf \(F\).
Recalling the twist by \(L^{r(m+1)}\) induced by \(W(s_{0}, \dotsc, s_{m})\), the previous equality means exactly that the Wronskian \(W(s_{0}, \dotsc, s_{m})\) induces, by evaluation on the foliation, a non-zero global section of 
\[
\left(
K_{\mathcal{F}}^{\frac{m(m+1)}{2}} \otimes L^{r(m+1)}
\right)_{\big\vert X \setminus \Sing(\mathcal{F})}.
\]
Since the previous coherent sheaf of rank one is reflexive, and since the codimension of \(\Sing(\mathcal{F})\) is at least \(2\), such global sections actually extend to global sections of  \(K_{\mathcal{F}}^{\frac{m(m+1)}{2}} \otimes L^{r(m+1)}\).
Now, letting \(r \to \infty\), this gives the pseudo-effectiveness of \(K_{\mathcal{F}}\) since \(\frac{r(m+1)}{m(m+1)} \sim \frac{1}{Cr}\).
\end{proof}
The previous proposition must be interpretated as a condition ensuring the algebraic integrability of a foliation by curves: namely, a  foliation by curves \(\mathcal{F}\) is algebraically integrable if its cotangent bundle (which is also the canonical bundle in the case of a foliation by curves) is not pseudo-effective. By the works of Bogomolov-McQuillan, Bost, Campana-Paun and Druel, this condition remains valid for higher dimensional foliations: see \cite{bogomolov2016rational}, \cite{Bostfol}, \cite{CampanaPaun}, \cite{Druel}.

The goal of this section is to adapt the argument of Proposition \ref{prop: fol rk 1} in the case of foliations of rank \(p>1\), using geometric generalized Wronskians. We are then able to prove a weak version of the previous algebraic integrability criterion, namely:
\begin{theorem}
\label{thm: foliation}
Let \(X\) be a normal projective variety, \(\mathcal{F}\) a  foliation of rank \(p \geq 1\) on \(X\) and \(L \to X\) an ample line bundle. Suppose that \(\mathcal{F}\) is not algebraically integrable. Then for any \(c \in \N_{\geq 1}\), there exists \(m, n \in \N_{\geq 1}\) satisfying \(m > nc\) such that the following holds
\[
H^{0}(X, \Omega_{\mathcal{F}}^{[m]} \otimes L^{n}) \neq \Set{0}.
\]
\end{theorem}
The strongest version, as stated in \cite{Druel}[Proposition 8.4], is obtained by replacing in the previous theorem the (reflexive) tensor product \(\Omega_{\mathcal{F}}^{[m]}\) by the (reflexive) symmetric product \(S^{[m]} \Omega_{\mathcal{F}}\).
\begin{proof}
Start by arguing as in Proposition \ref{prop: fol rk 1}. Polarize the projective variety \(X\) with the ample line bundle \(L\), and denote \(N=\dim(X)\). Since the general leaf of \(\mathcal{F}\) is not algebraic, there exists a constant \(C \in \N_{\geq 1}\) such that for any \(r \in \N_{\geq 1}\), there is at least \(Cr^{p+1}\) sections in \(H^{0}(X, L^{r})\) that are independant when restricted to a general leaf \(F\). 

For what follows, one works on the open set \(X \setminus \Sing(\mathcal{F})\).
Keeping the notations of the introduction of this section, denote \(m=m(n)=\abs{\mathcal{U}_{n}}\), and let \(r=r(n) \in \N\) be the smallest integer such that \(Cr^{p+1} > m\).
Let then \(s_{0}, \dotsc, s_{m}\) be \(m+1\) global sections of \(L^{r}\).

Denote \(\gamma\colon  (\C^{N-p},0) \times (\C^{p},0) \to X \) a local integration of the foliation around a point \(x \in X \setminus \Sing(\mathcal{F})\). Therefore, the leaves of the foliation are (locally) parametrized by \(\gamma_{\mathbi{t}}\colon (\C^{p},0) \to X\), where \(\mathbi{t}\) denotes coordinates on \((\C^{N-p},0)\). Two local integrations of the foliation differ by a family of biholomorphisms \(\varphi_{\mathbi{t}}\colon (\C^{p},0) \to (\C^{p},0)\).

This time, one has a large family of geometric generalized Wronskians that can be evaluated at the germ \(\gamma_{\mathbi{t}}\), namely
\[
W(s_{0}, \dotsc, s_{m})
\]
where \(W\) is any element in \(\mathcal{W}_{g}(m)\). To define an object that is intrinsic to the foliation, one proceeds as follows. Recall that we introduced in Section~\ref{sse: gg W and repr} a filtration of  \(\mathcal{W}_{g}(m)\) by caracteric sequences of \(m\).
Since the sections \(s_{0}, \dotsc, s_{m}\) are linearly independant when restricted to a general leaf \(L\) of \(\mathcal{F}\), Theorem \ref{thm: geometric generalized W} ensures that there exists at least one caracteristic sequence \(\mathbi{n}\) such that not all elements \(W(s_{0}, \dotsc, s_{m})\), \(W \in \mathcal{W}^{\leq \mathbi{n}}_{g}(m)\),  vanish identically when evaluated at \(\gamma_{\mathbi{t}}\). Fix then \(\mathbi{n}=(n_{1}, \dotsc, n_{k})\) minimal for this property, with respect to the ordering defined in Section~\ref{sse: gg W and repr}.

The key observation is then the following:
 for any \(W \in \mathcal{W}_{g}^{\leq \mathbi{n}}(m)\), the action of the biholomorphisms of \((\C^{p},0)\) on \(W(s_{0}, \dotsc, s_{m})(\gamma_{\mathbi{t}})\) reduces to the action of the linear group \(GL_{p}(\C)\). Indeed, using the formula Lemma \ref{lemma: compo} to compute multi-derivates of compositions of maps and the multi-linearity of the determinant, one sees that for \(\varphi\) a biholomorphism of \((\C^{p},0)\) and \(W \in \mathcal{W}^{\leq \mathbi{n}}_{g}(m)\),  \(\varphi \cdot W\) can be expressed as follows
\[
\varphi \cdot W
=
\diff \varphi(0)\cdot W
+
W',
\]
where \(W' \in \mathcal{W}_{g}^{<\mathbf{n}}(m)\). The very definition of \(\mathbi{n}\) gives then the equality 
\[
(\varphi_{\mathbi{t}} \cdot W)(s_{0}, \dotsc, s_{m})(\gamma_{\mathbi{t}}(\mathbi{z}+\cdot))
=
(\diff \varphi_{\mathbi{t}}(\mathbi{z})\cdot W)(s_{0}, \dotsc, s_{m})(\gamma_{\mathbi{t}}(\mathbi{z}+\cdot)).
\]

Now, pick one irreducible representation (of \(GL_{p}(\C)\)) \(\mathcal{V}\) of \(\mathcal{W}_{g}^{\leq \mathbi{n}}(m)\) whose elements do not all vanish identically when associated to \(s_{0}, \dotsc, s_{m}\) and evaluated at \(\gamma_{\mathbi{t}}\). By definition of \(\mathbi{n}\), one knows that \(\mathcal{V}\) is not included in \(\mathcal{W}_{g}^{< \mathbi{n}}(m)\). By what we saw at end of Section~\ref{sse: gg W and repr}, there exists a partition \(\lambda=(\lambda_{1}, \dotsc, \lambda_{r})\) of the number \(w\bydef w(\mathbi{n})\) such that
\[
\mathcal{V} \simeq S^{\lambda}\C^{p},
\]
where \(S^{\lambda}\) is the Schur functor associated to the partition \(\lambda\). Fix a basis \(\left(W_{T}\right)_{T \in \mathcal{T}}\) of \(\mathcal{V}\), where \(\mathcal{T}\) is the set of semi-standard tableaux of shape \(\lambda\), and consider the evaluation map at \(\gamma_{\mathbi{t}}(\mathbi{z}+\cdot)\):
\[
(\mathbi{t}, \mathbi{z})  
\mapsto 
\big(
W_{T}(s_{0}, \dotsc, s_{m})(\gamma_{\mathbi{t}}(\mathbi{z}+\cdot)
\big)_{T \in \mathcal{T}} \in \C^{\dim S^{\lambda}\C^{p}}.
\]
By very construction, if one reparametrizes the base-space by the family of diffeomorphisms \(\varphi_{\mathbi{t}}: (\C^{p},0) \to (\C^{p},0)\), one has the following equality
\[
\big(W_{T}(s_{0}, \dotsc, s_{m})\big)_{T \in \mathcal{T}}(\gamma_{\mathbi{t}} \circ \varphi_{\mathbi{t}}(\mathbi{z}+\cdot))
=
\big(\diff \varphi_{\mathbi{t}}(\mathbi{z}) \cdot W_{T}(s_{0}, \dotsc, s_{m})\big)_{T \in \mathcal{T}}(\gamma_{\mathbi{t}}(\mathbi{z}+\cdot)),
\]
 where, on the right, the action of the invertible linear map \(\diff \varphi_{\mathbi{t}}(\mathbi{z})\) on \(W_{T}\) is induced by the natural action of \(GL_{p}(\C)\) on \(S^{\lambda}\C^{p}\). In other words, the vector
 \[
 \left(W_{T}(s_{0}, \dotsc, s_{m})(\gamma_{\mathbi{t}}(\mathbi{z}+\cdot))\right)_{T \in \mathcal{T}}
 \]
 behaves exactly like an element of \(S^{\lambda}\Omega_{\mathcal{F}}\) written in the local trivialization \(\gamma\) adapted to the foliation \(\mathcal{F}\). Therefore, recalling the twist by \(L^{r(m+1)}\) induced by any geometric generalized Wronskian \(W(s_{0}, \dotsc, s_{m})\), one has constructed, by evaluating simultaneously the Wronskians \((W_{T}(s_{0}, \dotsc, s_{m}))_{T \in \mathcal{T}}\) at the foliation, a non-zero global section of 
 \[
 \left(S^{\lambda}\Omega_{\mathcal{F}} \otimes L^{r(m+1)}\right)_{\big\vert X \setminus \Sing(\mathcal{F})}.
 \] 
 In particular, this provides a non-zero global section of 
 \[
 \left(\Omega_{\mathcal{F}}^{\otimes w(\mathbi{n})} \otimes L^{r(m+1)}\right)_{\vert X \setminus \Sing(\mathcal{F})}
 =
 \left(\Omega_{\mathcal{F}}^{[w(\mathbi{n})]} \otimes L^{r(m+1)}\right)_{\vert X \setminus \Sing(\mathcal{F})},
 \]
 where the previous equality follows from the very definition of \(\Sing(\mathcal{F})\).
 As in Proposition \ref{prop: fol rk 1}, the previous coherent sheaf is reflexive, and since the codimension of \(\Sing(\mathcal{F})\) is at least \(2\), such global sections actually extend to global sections of  \(\Omega_{\mathcal{F}}^{[w(\mathbi{n})]} \otimes L^{r(m+1)}\).
 
 It now remains to evaluate the asymptotic behavior of the quotient
 \[
 \frac{r(n)(m(n)+1)}{w(\mathbi{n}(n))}
 \]
 for \(n \to \infty\), where we recall that \(m(n)=\abs{\mathcal{U}_{n}}\), \(r(n)\) is the smallest integer so that \(m(n) \geq Cr(n)^{p+1}\) and \(\mathbi{n}(n)\) is a caracteristic sequence of an admissible set of size \(m(n)\). To conclude, it is of course enough to show that this quotient tends to zero for \(n \to \infty\).
 Recall that the if the caracteristic sequence \(\mathbi{n}(n)\) writes \(\mathbi{n}(n)=(n_{1}(n), \dotsc n_{k(n)}(n))\), one has the following equalities
\[
\left\{
\begin{array}{rcr}
n_{1}(n)+2n_{2}(n)+ \dotsb + k(n)n_{k(n)}(n)
=
w(\mathbi{n}(n))
\\
n_{1}(n) + \dotsc + n_{k}(n)
=
m(n)
=
\abs{\mathcal{U}_{n}}
\end{array}
\right..
\]
Amongst the possible values of
\(
w(n)
=
n_{1}+2n_{2}+ \dotsb + k(n)n_{k(n)},
\)
where \((n_{1}, \dotsc, n_{k(n)})\) satisfies 
\begin{itemize}
\item{} \(n_{1} + n_{2} + \dotsb + n_{k(n)}=m(n) \);
\item{} \(n_{i} \leq \abs{\Set{\overline{w} \in \mathcal{W}_{p} \ | \ \ell(\overline{w})=i}}\) (recall that \(\mathbi{n}(n)\) is a caracteristique sequence!),
\end{itemize}
the minimal one \(w_{\min}(n)\)  is obtained by chosing successively for each \(n_{i}\) the maximal possible value. Therefore, from our very choice of \(m(n)=\abs{\mathcal{U}_{n}}\), one sees that
\[
w_{\min}(n)
=
w(\mathcal{U}_{n}).
\]
One can now conclude via the asymptotic estimates obtained in the introduction
\[
 \frac{r(n)(m(n)+1)}{w(\mathbi{n}(n))}
 \leq
 \frac{r(n)(m(n)+1)}{w_{\min}(n)}
 \sim
 r(n) \frac{\abs{\mathcal{U}_{n}}}{w(\mathcal{U}_{n})}
 \sim 
 \frac{p+1}{p} \frac{r(n)}{n} \to 0,
 \]
where the last limit follows from the estimate
\[
r(n) \leq \frac{1}{C} m(n)^{\frac{1}{p+1}} \sim C' n^{\frac{p}{p+1}},
\]
where \(C'\) is a constant.
\end{proof}

\subsection{Adjoint line bundles of foliated varieties.}
\label{sse: adjoint}
In the previous section, and especially in the proof of Theorem \ref{thm: foliation}, we used a stratification of all possible geometric generalized Wronskians of size \(\abs{\mathcal{U}_{n}}=\abs{\mathcal{W}_{p,\leq n}}\), \(n \in \N\), in order to discuss possible constructions of non-zero sections of the tensor algebra of a foliation of rank \(p\). The (unique) geometric generalized Wronskian belonging to the first step of the filtration is nothing but the Wronskian associated to the full set \(\mathcal{U}_{n}\), where we recall that by definition \(\mathcal{U}_{n} \bydef \mathcal{W}_{p, \leq n}\). (As in the previous section, the integer \(p\) is fixed). In this situation, the section we can produce with such a generalized Wronskian is straightforward:
\begin{lemma}
\label{lemma: Wronskian det}
For any \(\varphi: \C^{p} \to \C^{p} \) holomorphic map, and any \(f_{0}, \dotsc, f_{m}\) holomorphic functions, where \(m=\abs{\mathcal{U}_{n}}\), one has the equality
\[
W_{\mathcal{U}_{n}}(f_{0} \circ \varphi, \dotsc, f_{m} \circ \varphi)
=
\det(\diff \varphi)^{\frac{w(\mathcal{U}_{n})}{p}}W_{\mathcal{U}_{n}}(f_{0}, \dotsc, f_{m}).
\]
\end{lemma}
\begin{proof}
Using the formula Lemma \ref{lemma: compo} to compute multi-derivatives of composition of maps as well as the multi-linearity of the determinant, one easily sees that there exists a polynomial \(P\) in the \emph{first} derivatives of \(\varphi\) such that
\[
W_{\mathcal{U}_{n}}(f_{0} \circ \varphi, \dotsc, f_{m} \circ \varphi)
=
P(\partial_{1}\varphi, \dotsc, \partial_{p}\varphi) \times W_{\mathcal{U}_{n}}(f_{0}, \dotsc, f_{m}).
\]
This implies in particular that the line in \(\mathcal{W}_{g}(m)\) spanned by \(W_{\mathcal{U}_{n}}\) defines a one-dimensional representation of \(GL_{p}(\C)\): therefore, it must be a power of the determinant. On the other hand, if one takes \(\varphi= \Diag(\lambda_{1}, \dotsc, \lambda_{p})\), one has the following equality
\[
W_{\mathcal{U}_{n}}(f_{0} \circ \varphi, \dotsc, f_{m} \circ \varphi)
=
\left(\lambda_{1}\dotsb \lambda_{p}\right)^{\frac{w(\mathcal{U}_{n})}{p}}
W_{\mathcal{U}_{n}}(f_{0}, \dotsc, f_{m}).
\]
(Recall that the caracteristic exponent of \(\mathcal{U}_{n}\) is \(\boldsymbol{\beta}=(\frac{w(\mathcal{U}_{n})}{p}, \dotsc, \frac{w(\mathcal{U}_{n})}{p})\)). This finishes the proof of the lemma.
\end{proof}
Of course, we want to be able to produce non-zero sections, and to this end, we introduce the following definition:
\begin{definition}
Let \(X\) be a smooth projective variety of dimension \(N=\dim(X)\), \(\mathcal{F}\) a regular foliation of rank \(p \geq 1\)  and \(L \to X\) a line bundle. Let \(x \in X\) and let \((\mathbi{t},\mathbi{z}) \in (\C^{N-p},0) \times (\C^{p},0)\) be coordinates in a foliated trivializing chart around \(x\). For any \(n \in \N_{\geq 1}\), one says that the line bundle \(L\) \textsl{generates \(n\)-jets of the leaf at \(x\)} if and only if there exists sections \(s_{0}, \dotsc, s_{m}\) of \(H^{0}(X,L)\) such that \(s_{0}(0,\bigcdot), \dotsc, s_{m}(0, \bigcdot)\) generate \(\mathcal{O}_{(\C^{p},0)}/_{m_{0}^{n+1}}\).
\end{definition}
In the case where \(\mathcal{F}=TX\) is the trivial foliation, one simply says that \textsl{L generates \(n\)-jets at \(x\)} (see \cite{Laz1}[Definition 5.1.15]). Once this definition is introduced, the following lemma will allow us to produce non-zero sections (recall that the quotient \(\mathcal{O}_{(\C^{p},0)}/_{m_{0}^{n+1}}\) is isomorphic to the space of polynomials \(\C[Z_{1}, \dotsc, Z_{p}]_{\leq n}\) of degrees \(\leq n\)).
\begin{lemma}
\label{lemma: non-vanishing section}
Let \(m=\abs{\mathcal{U}_{n}}\).
If \(f_{0}, \dotsc, f_{m}\) are polynomials forming a basis of \(\C[Z_{1}, \dotsc, Z_{p}]_{\leq n}\), then
\(
W_{\mathcal{U}_{n}}(f_{0}, \dotsc, f_{m})
\)
is a non-zero constant.
\end{lemma}
\begin{proof}
This could follow from a direct computation, but one rather proceeds as follows. Observe first that if \(\mathcal{U}\) is a full set of size \(m\) distinct from \(\mathcal{U}_{n}\), then one necessarily has that
\[
W_{\mathcal{U}}(f_{0}, \dotsc, f_{m}) \equiv 0.
\]
Indeed, the full set \(\mathcal{U}\) must contain a word \(\overline{u}\) of length \(\ell(\overline{u}) > n\), so that partial derivative \(\partial_{\overline{u}}\) annihilates any of the polynomials \(f_{i}\) (as they are of degree \(\leq n\)). On the other hand, since the \(f_{i}\)'s are linearly independent, one knows from Theorem \ref{thm: geometric generalized W} that \(W_{\mathcal{U}_{n}}(f_{0}, \dotsc, f_{m})\) cannot be identically zero.

To conclude, it is enough to show that \(\partial_{i} W_{\mathcal{U}_{n}}(f_{0}, \dotsc, f_{m}) \equiv 0\) for any \(1 \leq i \leq p\). This follows immediately from the formula to compute the derivative of a determinant, and the fact that \(\partial_{\overline{u}}\) annihilates any of the \(f_{i}\)'s for \(\overline{u}\) a word of length \(\ell(\overline{u}) > n\).
\end{proof}

With these two lemmas at hand, we can now easily prove the following:

\begin{proposition}
\label{prop: Wronskian construction}
Let \(X\) be a smooth projective variety, \(\mathcal{F}\) a regular foliation of rank \(p \geq 1\)  and \(L \to X\) a line bundle. Let \(x \in X\), \(n \in \N_{\geq 1}\) and suppose that \(L\) generates \(n\)-jets of the leaf at \(x\). Then there exists a section of the line bundle
\[
\frac{1}{p} w(\mathcal{U}_{n}) K_{\mathcal{F}} 
+
(\abs{\mathcal{U}_{n}}+1)L
\]
that does not vanish at \(x\).
\end{proposition}
\begin{proof}
Denote \(N=\dim(X)\). Let \((\C^{N-p},0)\times (\C^{p},0) \overset{\gamma}{\simeq} U_{x}\) be a trivializing open set of the foliation around \(x\), and denote \((\mathbi{t}, \mathbi{z})\) the coordinates on \((\C^{N-p},0)\times (\C^{p},0)\). Denote also \(m= \abs{\mathcal{U}_{n}}\). By hypothesis, one can pick \(s_{0}, \dotsc, s_{m}\) global sections of \(L\) such that the reduction of
 \[
 (s_{0}(\gamma(0,\mathbi{z})), \dotsc, s_{m}(\gamma(0,\mathbi{z}))
 \]
 modulo \(\left(m_{(\C^{p},0)}\right)^{n+1}\) forms a basis of \(\C[Z_{1}, \dotsc, Z_{p}]_{\leq n}\). In particular, by Lemma \ref{lemma: non-vanishing section}, the function
 \[
 W_{\mathcal{U}_{n}}\left(s_{0}(\gamma(0,\mathbi{z})), \dotsc, s_{m}(\gamma(0,\mathbi{z}))\right)
 \]
 does not vanish at \(\mathbi{z}=0\).
 
 Once this is established, one argues as in Theorem \ref{thm: foliation} to produce the wanted section. A change of chart for the foliation is caracterized by a biholomorphism 
 \[
 (\psi, \varphi): (\C^{N-p},0)\times (\C^{p},0) \to (\C^{N-p},0)\times (\C^{p},0),
 \]
with \(\psi\) depending on \(\mathbi{t}\) only. By Lemma \ref{lemma: non-vanishing section}, the numbers 
\[
W_{\mathcal{U}_{n}}
\big(
s_{0}(\gamma(\mathbi{t},\mathbi{z})), \dotsc, s_{m}(\gamma(\mathbi{t},\mathbi{z}))
\big)
\]
and
\[
W_{\mathcal{U}_{n}}
\big(
s_{0}(\gamma(\psi(\mathbi{t}),\varphi(\mathbi{t},\mathbi{z})), \dotsc, s_{m}(\gamma(\psi(\mathbi{t}),\varphi(\mathbi{t},\mathbi{z}))
\big)
\]
differ by multiplication by the number 
\(\left(
\det(\diff_{\mathbi{z}} \varphi_{\mathbi{t}})
\right)^{\frac{w(\mathcal{U}_{n})}{p}}\). Implicit in the previous writing is a choice of trivialization for the line bundle \(L\), and one knows by standard properties of geometric generalized Wronskians that one passes from a trivialization to another by multiplication by the transition maps of the line bundle elevated to the \(\left(\abs{\mathcal{U}_{n}}+1\right)\) power. 

All in all, this shows that one has defined a global section of 
\[
K_{\mathcal{F}}^{\frac{w(\mathcal{U}_{n})}{p}}
\otimes 
L^{\abs{\mathcal{U}_{n}}+1}
\]
that does not vanish at \(x\), which proves the proposition.
\end{proof}

To discuss further the corollaries of this proposition, we are naturally lead to introduce the following definition, generalizing a definition of Demailly \cite{demailly1992singular} to the case of foliated varieties:
\begin{definition}
Let \(L\) be a line bundle on a smooth projective variety \(X\), and let \(\mathcal{F}\) be a regular foliation on \(X\). For any \(x \in X\), define
\[
s(L, \mathcal{F};x)
\bydef
\max \Set{r \ | \ \text{\(L\) generates \(r\)-jets of the leaf at \(x\)}}.
\]
If the previous set is empty, set \(s(L, \mathcal{F};x)=-1\).
Define further
\[
\sigma(L, \mathcal{F};x)
\bydef
\limsup_{k} \frac{s(kL, \mathcal{F}; x)}{k},
\]
as well as 
\[
\sigma(L, \mathcal{F})
\bydef 
\inf_{x \in X}
\sigma(L, \mathcal{F};x).
\]
\end{definition}
In the case where \(\mathcal{F}=TX\) and \(L\) is ample, Demailly has shown that \(\sigma(L, TX;x)\) is equal to the Seshadri constant \(\varepsilon(L;x)\) at \(x\): see \cite{demailly1992singular}, and see \cite{Laz1}[Section 5] for definitions and further details on the Seshadri constant.

With these definitions at hand, we deduce almost immediately from the previous Proposition \ref{prop: Wronskian construction} the following corollary:
\begin{corollary}
Let \(X\) be a smooth projective variety, \(\mathcal{F}\) a regular foliation of rank \(p \geq 1\)  and \(L \to X\) a nef line bundle. 
Suppose that there exists a constant \(K >0\) such that the following inequality holds outside a countable union of points:
\[
\sigma(L,\mathcal{F};x) \geq K.
\]
Then the line bundle \(K_{\mathcal{F}}+\frac{(p+1)}{K}L\) is nef.
\end{corollary}
\begin{proof}
Let \(C\) be any irreducible curve in \(X\), and let \(x \in C\) be a point outside the countable union of points of the statement. Denote \(n(k)\bydef s(kL, \mathcal{F};x) \), so that by hypothesis \(\limsup_{k} \frac{n(k)}{k} \geq K\). By Propososition \ref{prop: Wronskian construction} applied to the line bundle \(L^{k}\), one immediately deduces that the following inequality holds:
\[
\left(\frac{1}{p} w(\mathcal{U}_{n(k)}) K_{\mathcal{F}} 
+
k(\abs{\mathcal{U}_{n(k)}}+1)L\right) \cdot C
\geq 0.
\]
Now, as for \(k \to \infty\) (or rather an extracted sequence) one has the following estimate
\[
\frac{n(k)(\abs{\mathcal{U}_{n(k)}}+1)}{\frac{1}{p} w(\mathcal{U}_{n(k)})}
\sim
(p+1)\frac{k}{n(k)},
\]
one deduces that \(\left(K_{\mathcal{F}} +\frac{(p+1)}{K}L\right)\cdot C \geq 0\) (recall that \(L\) is nef). As this holds for any curve, this proves the result.
\end{proof}
We can specialize the previous corollary to the case where \(\mathcal{F}=TX\) is the trivial foliation, and \(L\) is an ample line bundle. (Recall that in this case, the constant \(\sigma(L, TX;x)\) is nothing but the Seshadri constant \(\varepsilon(L;x)\) of \(L\) at \(x\) by a Theorem of Demailly \cite{demailly1992singular}).
\begin{corollary}
\label{cor: cor2}
Let \(X\) be a smooth projective variety and \(L \to X\) an ample line bundle. 
Suppose that there exists a constant \(K >0\) such that the following inequality holds outside a countable union of points:
\[
\varepsilon(L;x) \geq K.
\]
Then the line bundle \(K_{X}+\frac{(\dim(X)+1)}{K}L\) is nef.
\end{corollary}
Note that this gives an elementary way to recover the upper bound of the Seshadri constant of the anti-canonical bundle of a Fano manifold \(X\) (i.e. a smooth projective variety \(X\) with \(-K_{X}\) ample):
\[
\varepsilon(X, -K_{X})
\leq 
\dim(X)+1.
\]
See e.g. \cite{bauer2009seshadri} where the following stronger result is shown: \(\varepsilon(X, -K_{X})=\dim(X)+1\) if and only if \(X=\P^{\dim(X)}\), and otherwise \(\varepsilon(X, -K_{X}) \leq \dim(X) \).

Adjoint linear systems of the form \(K_{X} +  mL\), with \(m \in \N_{\geq 1}\) and \(L\) an ample line bundle, have been extensively studied in the litterature, motivated by Fujita's conjectures, that we recall:
\begin{itemize}
\item{(Fujita's freeness conjecture)}  \(K_{X}+(\dim(X)+1)L\) is globally generated;
\item{(Fujita's very ampleness conjecture)} \(K_{X}+ (\dim(X)+2)L\) is very ample.
\end{itemize}
Fujita's freeness conjecture was proved for \(\dim(X)=2, 3, 4, 5\) by respectively \cite{reider1988vector}, \cite{ein1993global}, \cite{kawamata1995fujita} and \cite{ye2020fujita}. It holds when \(L\) is free (and of course ample) thanks to Kodaira vanishing theorem and Castelnuovo-Mumford regularity (see \cite{Laz1}[1.8]). Without any hypothesis  on \(L\) (beside ampleness), an easy corollary of Mori's cone Theorem implies that \(K_{X}+ (\dim(X)+1)L\) is nef. 

The previous Corollary \ref{cor: cor2} would give an alternate proof of the fact that \(K_{X} + (\dim(X)+1)L\) is nef provided one can prove that for any ample line bundle \(L\), one has the following estimate
\[
\varepsilon(L;x) \geq 1
\]
outside a countable union of points. It turns out to be true when \(X\) is a surface (see \cite{Laz1}[Proposition 5.3]), and the conjecture appearing in the litterature regarding Seshadri constants of ample line bundles asserts that the inequality
\(
\varepsilon(L;x) \geq 1
\) 
should hold for \(x\) very general. In any event, it has been proved that the following inequality holds
\[
\varepsilon(L;x) \geq \frac{1}{\dim(X)}
\]
for \(x\) very general (see \cite{Laz1}[Theorem 5.2.5]).

Returning to the case of foliated varieties, it is naturally tempting to state the following conjectures:
\begin{conjecture}
Let \(\mathcal{F}\) be a regular foliation of rank \(p\) on a smooth projective variety \(X\). Then:
\begin{itemize}
\item{} \(K_{\mathcal{F}} + (p+1)L\) is globally generated;
\item{} \(K_{\mathcal{F}} + (p+2)L\) is very ample.
\end{itemize}
\end{conjecture}
In view of the results used to prove partial results motivating Fujita's conjecture, we are also naturally lead to address the following questions:
\begin{question}
For \(X\) a smooth projective variety equipped with a foliation \(\mathcal{F}\), and \(L\) an ample line bundle on \(X\), do we have a vanishing theorem alike Kodaïra's vanishing theorem? Especially, do we have \(H^{1}(X, K_{\mathcal{F}}+L)=0\)?
\end{question}

\begin{question}
In the same setting as before, does there exist a ``Mori cone Theorem ``, where \(K_{X}\) is replaced by \(K_{\mathcal{F}}\)?
\end{question}
Regarding this question, let us note that it has a positive answer in the case of foliations by curves: see \cite{bogomolov2016rational}. Furthermore, in \cite{spicer2017higher}, a cone theorem for codimension one foliations on \(3\)-folds is established.

\newpage

\appendix

\section{Formulas to compute partial derivatives of products and compositions}
\label{appendix: formula}
In this appendix, we gather  two formulas that are repeatedly used in the text: a Leibniz rule to compute higher order partial derivatives of product of functions, and a formula to compute higher order partial derivatives of composition of maps. Both formulas are proved by a straightforward induction, using in both cases the well-known formulas for the first derivatives.

The Leibniz rule reads as follows:
\begin{lemma}[Leibniz rule]
\label{lemma: Leibniz}
Let \(f, g: (\C^{p},0) \rightarrow \C\) be holomorphic functions. 
There exists universal constants \(C_{(\overline{a},\overline{b})} \in \N_{\geq 1}\), where \((\overline{a}, \overline{b})\) runs over \(\mathcal{W}_{p}^{2}\), such that for any word \(\overline{u} \in \mathcal{W}_{p}\), one has the equality
\[
\partial_{\overline{u}}(fg)
=
\sum\limits_{\overline{u_{1}}\cdot \overline{u_{2}}=\overline{u}} C_{(\overline{u_{1}}, \overline{u_{2}})} \partial_{\overline{u_{1}}}f \partial_{\overline{u_{2}}}g,
\]
where the sum symbol runs over the decomposition of the word \(\overline{u}\) into two words such that
\begin{itemize}
\item{}
the order is taken into account;
\item{}
 the empty word is allowed.
\end{itemize}

The previous constants are defined by the following formula:
\[
C_{\overline{a}, \overline{b}}=\prod\limits_{i=1}^{p}\binom{\alpha_{i}(\overline{a})+\alpha_{i}(\overline{b})}{\alpha_{i}(\overline{a})}.
\]
\end{lemma}
\begin{proof}
We argue by induction on \(p \geq 1\). For \(p=1\), this is the usual Leibniz rule, as $\overline{u}$ is of the form \((1^{k}), k \geq 1\), and thus \(\partial_{\overline{u}}=(\frac{d}{dz})^{k}\). Let then \(p > 1\), and write 
\[
\overline{u}=\overline{v} \cdot (p^{k}),
\]
where \(k=\alpha_{p}(\overline{u})\).
If \(k=0\), then the result follows by induction. Suppose therefore that \(k \geq 1\), and write
\begin{equation*}
\begin{aligned}
\partial_{\overline{u}}(fg)
&
=
\partial_{\overline{v}}\big(\partial_{(p^{k})}(fg)\big)
\\
&
=
\partial_{\overline{v}}\Big(\sum\limits_{i=0}^{k} \binom{k}{i} \partial_{(p^{i})}(f)  \partial_{(p^{k-i})}(g) \Big)
\\
&
=
\sum\limits_{i=0}^{k} \binom{k}{i} \partial_{\overline{v}}\big(\partial_{(p^{i})}(f) \partial_{(p^{k-i})}(g)\big),
\end{aligned}
\end{equation*}
where for the second equality one uses the usual Leibniz rule.
By induction, one deduces that
\[
\partial_{\overline{u}}(fg)
=
\sum\limits_{i=0}^{k} \sum\limits_{\overline{v_{1}}\cdot\overline{v_{2}}=\overline{v}}  
\binom{k}{i} 
C_{(\overline{v_{1}}, \overline{v_{2}})} 
\partial_{\overline{v_{1}}.(p^{i})}(f) \partial_{\overline{v_{2}}.(p^{k-i})}(g),
\]
so that
\[
C_{(\overline{v_{1}}.(p^{i}), \overline{v_{2}}.(p^{k-i}))}
=
\binom{k}{i} C_{\overline{v_{1}}, \overline{v_{2}}},
\]
which concludes the proof.
\end{proof}

As for the formula to compute higher order partial derivatives of composition of maps, it reads as follows. Note that we have not given an explicit formula for the constants as in the Leibniz rule: it seems to be a much more complicated combinatorial problem to do so.

\begin{lemma}[Derivative of composition of maps]
\label{lemma: compo}
Let \(f: \C^{n} \to \C\) be a holomorphic function and let \(\varphi: \C^{p} \to \C^{n}\) be a holomorphic map. There exists universal constants \(\left(D_{\overline{a_{1}}, \dotsc, \overline{a_{k}}}(n)\right)_{k \in \N, \overline{a_{i}} \in \mathcal{W}_{p}}\) in \(\N\) such that for any word \(\overline{u} \in \mathcal{W}_{p}\), one has the following formula to compute the \(\overline{u}\)-partial derivative of compositions of functions:
\[
\partial_{\overline{u}}\left(f\circ \varphi\right)
=
\sum\limits_{k}
\sum\limits_{\overline{u_{1}} \dotsc \overline{u_{k}} = \overline{u}}
D_{\overline{u_{1}}, \dotsc, \overline{u_{k}}}(n)
\sum\limits_{i_{1}=1, \dotsc, i_{k}=1}^{n}
\partial_{\overline{i_{1} \dotsc i_{k}}}f \circ \varphi \times\partial_{\overline{u_{1}}} \varphi_{i_{1}} \times \dotsc \times \partial_{\overline{u_{k}}} \varphi_{i_{k}},
\]
where the second sum symbol runs over the decompositions of the word \(\overline{u}\) into \(k\) words such that
\begin{itemize}
\item{} the order is not taken into account (i.e. two decompositions are considered the same if one is obtained from the other by a permutation of the words);
\item{} the empty word is \textsl{not} allowed (i.e. one considers proper decompositions into \(k\) words).
\end{itemize}
\end{lemma}
\begin{proof}
Proceed by induction on the length of the word \(\overline{u}\), the case of \(\ell(\overline{u})=1\) being the usual formula to compute the derivative of composition of maps, with \(D_{\overline{1}}(n)= \dotsb=D_{\overline{p}}(n)=1\).

 Let therefore \(\overline{u}\) be a word of length \(\ell > 1\), and suppose that that it writes  \(\overline{u}=\overline{u'}\cdot \overline{j}\) for some \(1 \leq j \leq p\). By induction, one can write the following:
\[
\partial_{\overline{u}}\left(f\circ \varphi\right)
=
\sum\limits_{k}
\sum\limits_{\overline{u_{1}} \dotsc \overline{u_{k}} = \overline{u'}}
D_{\overline{u_{1}}, \dotsc, \overline{u_{k}}}(n)
\sum\limits_{i_{1}, \dotsc, i_{k}}
\partial_{j}\left(
\partial_{\overline{i_{1} \dotsc i_{k}}}f \circ \varphi \times\partial_{\overline{u_{1}}} \varphi_{i_{1}} \times \dotsc \times \partial_{\overline{u_{k}}} \varphi_{i_{k}}
\right).
\]
Each term appearing on the right can be rewritten as the sum of
\[
\sum\limits_{i=1}^{n} \partial_{\overline{i_{1} \dotsc i_{k}i}}f \circ \varphi \times \partial_{j}\varphi_{i} \times\partial_{\overline{u_{1}}} \varphi_{i_{1}} \times \dotsc \times \partial_{\overline{u_{k}}} \varphi_{i_{k}}
\]
and
\[
\partial_{\overline{i_{1} \dotsc i_{k}}}f \circ \varphi\left(
\sum\limits_{\ell=1}^{k}
 \partial_{\overline{u_{1}}} \varphi_{i_{1}} \times \dotsc \times \partial_{\overline{u_{\ell}}\cdot \overline{j}} \varphi_{i_{\ell}} \times \dotsc \times \partial_{\overline{u_{k}}} \varphi_{i_{k}}
\right).
\]
Therefore, one can indeed write \(\partial_{\overline{u}}\left(f\circ \varphi\right)\) under the following form:
\[
\partial_{\overline{u}}\left(f\circ \varphi\right)
=
\sum\limits_{k}
\sum\limits_{\overline{u_{1}} \dotsc \overline{u_{k}} = \overline{u}}
D_{\overline{u_{1}}, \dotsc, \overline{u_{k}}}(n)
\sum\limits_{i_{1}, \dotsc, i_{k}}
\partial_{\overline{i_{1} \dotsc i_{k}}}f \circ \varphi \times\partial_{\overline{u_{1}}} \varphi_{i_{1}} \times \dotsc \times \partial_{\overline{u_{k}}} \varphi_{i_{k}}
\]
for some constants \(D_{\overline{u_{1}}, \dotsc, \overline{u_{k}}}(n)\) in \(\N\).
\end{proof}

\newpage 

\section{A vanishing theorem}
\label{appendix: vanishing}
This appendix is devoted to the proof of the vanishing Theorem \ref{thm: vanishing}.
First, we introduce notations and results from Nevanlinna theory. For details, we refer to \cite{Nevan}.

\begin{definition} 
Let \(X\) be a complex manifold, and \(L \rightarrow X\)  an hermitian line bundle on \(X\) with metric \(h\). The \textsl{characteristic function} of a holomorphic map  \(f: \C^{p} \rightarrow X\) is defined as follows
\[
T_{f}(L,r)
\bydef
\int_{0}^{r} \frac{dt}{t} \int_{\B^{p}_{t}} f^{*}(c_{h})\wedge \omega_{0}^{p-1},
\]
where \(\B^{p}_{t}\) is the ball of radius \(t\) in \(\C^{p}\), \(c_{h}\) is the Chern form of the metric \(h\), and \(\omega_{0}=dd^{c} log ||Z||^2\) is the homogeneous metric form on \(\C^{p}\), where \(Z=(z_{1}, \dotsc, z_{p})\) are the canonical coordinates.
\end{definition}
One is mainly interested in the asymptotic behavior of \(T_{f}(L,r)\) for \(r \rightarrow \infty\), which does not depend on the metric \(h\) (see \cite{Nevan}[p.38]). This behavior dictates the rationality of the holomorphic map \(f\) when \(X\) is a projective variety, and \(L\) an ample line bundle (see \cite{Nevan}[p.66]):
\begin{theorem}
\label{thm: rationality}
A holomorphic map $f: \C^{p} \rightarrow \P^{N}$ is rational if and only if $T_{f}(L,r)=O(\log(r))$.
\end{theorem}

The following theorem is the several variables version of the so-called lemma on the logarithmic derivative (see \cite{dlog}):
\begin{theorem}
\label{thm: dlog}
Let $g: \C^{p} \rightarrow \P^{1}$ be a meromorphic map, and let $\overline{u}$ be a word in $\mathcal{W}_{p}$. There exists constants $a_{0}, a_{1}, a_{2}$ in $\R_{+}$ such that
\[
\int_{\partial B_{r}^{p}} \log^{+} |\frac{\partial_{\overline{u}}G}{G}| \cdot \sigma
\leq
a_{0} + a_{1}\log(r) + a_{2}\log T_{g}(L,r)
\
 ||,
\]
where \(G=\frac{g_{1}}{g_{0}}\) is \(g\) in its inhomogeneous form, and \(\sigma=d^{c}\log ||Z||^2 \wedge \omega_{0}^{p-1}\) is the Poincaré form on \(\C^{p}\).
\end{theorem}
Recall that the notation ``\(||\)`` means that the equality is satisfied outside a measurable set of finite Lebesgue measure, and that \(\log^{+}\bydef \max(0, \log)\). From the previous theorem, one deduces the following corollary:
\begin{corollary}
\label{cor: dlog}
Let \(g: \C^{p} \rightarrow \P^{1}\) be a meromorphic map, and let \(\overline{u}\) be a word in \(\mathcal{W}_{p}\). Then 
\[
\int_{\partial B_{r}^{p}} \log^{+} |\partial_{\overline{u}} \log G| \cdot \sigma
\underset{r \rightarrow \infty}{=}
O\big(\log(r)+\log T_{g}(L,r)\big).
\]
where \(G=\frac{g_{1}}{g_{0}}\) is \(g\) in its inhomogeneous form.
\end{corollary}
\begin{proof}
Once one observes that \(\partial_{\overline{u}} \log G \) can be written
\[
\partial_{\overline{u}} \log G
=
\frac{\partial_{\overline{u}} G}{G} + \sum\limits_{\ell(\overline{u_{1}}) + \dotsb + \ell(\overline{u_{k}}) < \ell(\overline{u})} C_{(\overline{u_{1}}, \dotsc, \overline{u_{k}})} \frac{\partial_{\overline{u_{1}}}G \dotsc \partial_{\overline{u_{k}}}G}{G^{k}}
\]
where \(C_{(\overline{u_{1}}, \dotsc, \overline{u_{k}})} \in \Z\) are constants independant of \(G\), the corollary then follows from a direct application of the previous theorem.
\end{proof}

The remaining of the appendix is devoted to the proof of the vanishing Theorem \ref{thm: vanishing}:
\begin{theorem}
Let \(X\) be a projective manifold equipped with an ample line bundle \(L\), and suppose that there exists a (non-zero) global section
\[
P \in H^{0}(X, E_{p,k,w}X\otimes L^{-1}),
\]
where \(p,k,w \in \N_{\geq 1}\). Then any non-degenerate map \(f \colon \C^{p} \rightarrow X \) satisfies the differential equation
\[
P(f)=P_{f}(\partial_{p,k}f)\equiv0.
\]

\end{theorem}
The exposition follows very closely the proof in one variable as described in \cite{Dem}, with suitable adaptations for the several variables setting. We could write down the proof by refering to the one in \cite{Dem}, and indicate the changes; however, since the paper is written in french, we give details for the reader's convenience.
We state and prove three lemmas, and then proceed to the proof.

\begin{lemma}
\label{lemma: lem1}
Let \(u: Y \dashrightarrow Z\) be a rational map between two projective varieties, and let \(\omega\) and \(\omega'\) be hermitian metrics on \(Y\) and \(Z\) respectively. There exists a constant \(C>0\) such that for any non-degenerate holomorphic map 
\(f: \C^{p} \rightarrow Y\)
whose image is not included in the indeterminacy locus of \(u\), the following inequality is satisfied
\[
T_{u \circ f}(L,r)
\leq 
CT_{f}(L,r) + O(1)
\]
where the constant \(O(1)\) might depend on \(f\).
\end{lemma}
\begin{proof}
One can always suppose that \(X\) (resp. \(Y\)) is embedded in a projective space \(\P^{N}\) (resp. \(\P^{N'}\)), and that \(\omega\) (resp. \(\omega'\)) is the restriction of the Fubini-Study metric on \(X\) (resp. \(Y\)). Let \((u_{0}, \dotsc, u_{N'})\) be rational functions of degree \(0\) defining \(u\) in restriction to \(X\), and write \(u_{j}=\frac{p_{j}}{q}\), where \(q\) and the \(p_{j}\)'s are polynomials with same degree \(m\).

Outside the indeterminacy locus of \(u\), one can write
\[
u^{*}\omega'=m\omega + \frac{i}{2\pi} \partial\overline{\partial} \psi,
\]
where \(\psi(\mathbf{z})
\bydef 
\log
\left(
\sum\limits_{j=0}^{N} \frac{\abs{p_{j}(\mathbf{z})}}{\abs{\mathbf{z}}^{2m}}
\right)
\)
is an upper-bounded function on \(\P^{N}\). Therefore, one deduces that 
\[
(u \circ f)^{*}\omega'
=
mf^{*}\omega + \frac{i}{2\pi} \partial\overline{\partial}(\psi \circ f).
\]
Using a generalization of Jensen's formula, see e.g \cite{Nevan}[p.29, Lemma 2], one deduces the result.
\end{proof}

\begin{lemma}
\label{lemma: lem2}
Let \(X\) be a projective manifold equipped with an ample line bundle \(L\), and suppose that there exists a non-zero global section
\[
P \in H^{0}(X, E_{p,k,w}X\otimes L^{-1}),
\]
where \(p,k,w \in \N_{\geq 1}\). Then any non-degenerate holomorphic map  \(f\colon \C^{p} \rightarrow X\) satisfies the following estimate:
\[
T_{f}(L,r) \underset{r \rightarrow \infty}{=} O(\log(r)).
\]
\end{lemma}
\begin{proof}
Denote \(\pi: J_{p,k}X \rightarrow X\) the canonical projection as well as \(\tilde{L}=\pi^{*}L\) the pull-back line bundle on \(J_{p,k}X\). Note that the global section \(P \in H^{0}(X, E_{p,k,w}X\otimes L^{-1})\) induces a section 
\[
s: J_{p,k}X \rightarrow \tilde{L}^{-1}, \ \ (x, \partial_{p,k}\gamma) \mapsto \big((x, \partial_{p,k}\gamma), P_{x}(\partial_{p,k} \gamma)\big).
\]
Note also that the holomorphic map \(f: \C^{p} \rightarrow X\) induces a holomorphic map
\[
f_{[k]}: \C^{p} \rightarrow J_{p,k}X
\]
satisfying the following equality: \(f_{[k]}^{*}\tilde{L}=f^{*}L\). 
As a corollary of Poincaré-Lelong formula, one obtains the following equality in the sense of currents (see \cite{Nevan}[p.54]):
\[
dd^{c} \log ||f_{[k]}^{*}s||_{f^{*}h} = -c_{f^{*}h^{-1}} + \Div(f_{[k]}^{*}s)
\]
where \(h\) is any positive metric on \(L\). By functoriality of the Chern class, this can be rewritten as follows
\[
dd^{c} \log ||f_{[k]}^{*}s||_{f^{*}h} = f^{*}c_{h} + \Div(f_{[k]}^{*}s).
\]
In particular, as the current associated to a divisor is always positive, one has the inequality (in the sense of currents)
\[
dd^{c} \log ||f_{[k]}^{*}s||_{f^{*}h} \geq f^{*}c_{h}.
\]
It implies that
\begin{equation*}
\begin{aligned}
T_{f}(L,r)
&
=
\int_{0}^{r} \frac{dt}{t} \int_{\B^{p}_{t}} f^{*}(c_{h})\wedge \omega_{0}^{p-1}
\\
&
\leq
\int_{0}^{r} \frac{dt}{t} \int_{\B^{p}_{t}} dd^{c} \log ||f_{[k]}^{*}s||_{f^{*}h}\wedge \omega_{0}^{p-1}
\\
&
=
\frac{1}{2} \int_{\partial \B_{r}^{p}} \log ||f_{[k]}^{*}s||_{f^{*}h} \cdot \sigma + O(1)
\end{aligned}
\end{equation*}
where one uses a generalization of Jensen's formula (see \cite{Nevan}[p.29, Lemma 2]) for the last equality.

As \(X\) is projective, one can embedd it in a projective space \(\P^{N}\), and consider the following finite family of rational maps
\[
\big(v_{i,j}=\frac{X_{i}}{X_{j}}\big)_{i\neq j}
\]
where \((X_{0}, \dotsc, X_{N})\) are the homogeneous coordinates on \(\P^{N}\).
Let \(x_{0} \in X\). As \(X\) is smooth, one can extract \(n=\dim(X)\) rational maps from the previous family such that they form local coordinates of \(X\) around \(x_{0}\). One we can always suppose that none of them vanishes around \(x_{0}\), and without loss of generality, suppose that this family is \((v_{1,0}, \dotsc, v_{n,0})\). By taking a local determination of their logarithm, observe that
\[
\big(\log (v_{1,0}), \dotsc, \log(v_{n,0})\big)
\]
remains local coordinates around \(x_{0} \in X\).
Writing \(f_{[k]}^{*}s\) in these local coordinates, one sees that it is expressed as a polynomial in the derivatives 
\[
\Big(\partial_{\overline{u}}(\log(v_{i,0} \circ f))\Big)_{\substack{1 \leq i \leq n \\ \overline{u} \in \mathcal{W}_{p, \leq k}}}
\]
with holomorphic coefficients in the variable \(f\), i.e.
\[
f_{[k]}^{*}s
\overset{loc}{=}
Q\big(f, \partial_{\overline{u}}(\log(v_{i,0} \circ f)) \big),
\]
where \(Q\big(z, \bold{w}=(w_{\overline{u},i}\big)_{\overline{u},i})= \sum\limits_{\boldsymbol{\alpha}} c_{\boldsymbol{\alpha}}(z) \bold{w}^{\boldsymbol{\alpha}}\), with \(c_{\boldsymbol{\alpha}}\) holomorphic.
From this equality, one infers that there exists a constant \(K\) such that the following inequality is satisfied on a neighborhood \(U_{x_{0}}\) of \(x_{0}\):
\[
 \log ||f_{[k]}^{*}s||_{f^{*}h} 
 \leq
 K
\sum\limits_{1 \leq i \leq n} \sum\limits_{\overline{u} \in \mathcal{W}_{p, \leq k}} \log^{+} | \partial_{\overline{u}}(\log(v_{i,0} \circ f)) |.
 \]

One can do the previous reasoning around every point of $X$, and by compacity of \(X\), one deduces that there exists a constant \(\overline{K}\) and a family of indexes 
\[
I=\{(1,0), \dotsc, (n,0), \dotsc \} \subset \{0, \dotsc, N\}^{2}
\]
such that for all \(x \in X\), the following inequality is satisfied:
\[
 \log ||f_{[k]}^{*}s||_{f^{*}h} 
 \leq
 \overline{K}
\sum\limits_{(i,j) \in I} \sum\limits_{\overline{u} \in \mathcal{W}_{p, \leq k}} \log^{+} | \partial_{\overline{u}}(\log(v_{i,j} \circ f)) |.
 \]
Now, for each \((i,j) \in I\), \(v_{i,j} \circ f\) can be seen as the inhomogeneous form of a meromorphic map \(f_{i,j}: \C^{p} \rightarrow \P^{1}\), so that the (corollary of) the lemma on the logarithmic derivative applies. One therefore gets the following inequalities:
\begin{equation*}
\begin{aligned}
T_{f}(L,r)
&
\leq
O(1) 
+
\frac{\overline{K}}{2} 
\sum\limits_{(i,j) \in I} \sum\limits_{\overline{u} \in \mathcal{W}_{p, \leq k}}
 \int_{\partial \B_{r}^{p}}
\log^{+} | \partial_{\overline{u}}(\log(v_{i,j} \circ f)) |
\cdot
\sigma
\\
&
\leq O(\log(r)) + O\left(\sum\limits_{(i,j) \in I}T_{v_{i,j}\circ f}(r)\right) 
\ 
||.
\end{aligned}
\end{equation*}
A straightorward application of Lemma \ref{lemma: lem1} finishes the proof.
\end{proof}

\begin{lemma}
\label{lemma: lem3}
For any \(p,k,w \in \N_{\geq 1}\), the following vanishing result holds:
\[
H^{0}(\P^{p}, E_{p,k,w}\P^{p}\otimes \O_{\P^{p}}(-1))
=0.
\]
\end{lemma}
\begin{proof}
Using the filtration described at the end of Section~\ref{ssse: jets differentials}, one infers that it is enough to prove the following vanishing result for any \(m \in \N_{\geq 1}\):
\[
H^{0}(\P^{p}, \Omega_{\P^{p}}^{\otimes m} \otimes \O_{\P^{p}}(-1))
=
0.
\]
This follows from a general vanishing theorem (see \cite{Dem}), but in our situation, one can argue by induction with explicit computations (using the Euler sequence for \(\Omega_{\P^{p}}\) and the computation of the cohomology of the line bundles \(\O_{\P^{p}}(r)\), \(r \in \Z\)).
\end{proof}
We can now prove the wanted vanishing theorem:

\begin{proof}[Proof of Theorem \ref{thm: vanishing}]
Argue by contradiction, and suppose that there exists a non-degenerate holomorphic map \(f: \C^{p} \rightarrow X\) such that \(P_{f}(\partial_{p,k}f) \not\equiv 0\). By Lemma \ref{lemma: lem2} and Theorem \ref{thm: dlog},  the map \(f\) is actually rational, so that it is induced by a rational map
\[
F: \P^{p} \dashrightarrow X,
\]
which is well defined outside a set \(I_{F}\) of codimension \(\geq 2\). One can pull-back the global section \(P\) to obtain a non-zero global section
\[
F^{*}P \in H^{0}(\P^{p} \setminus I_{F}, E_{p,k,w}\P^{p}\otimes F^{*}L^{-1}).
\]
Note that \(F\) is well-defined at the level of Picard groups (as it is well defined in codimension \(\geq 2\)), so that \(F^{*}L\) is indeed a line bundle on \(X\). Furthermore, observe that it must be a positive multiple of \(\mathcal{O}(1)\), hence very ample. Indeed, \(\Pic(\P^{p})=\Z\cdot \mathcal{O}(1)\), and \(F^{*}L\) has non-zero global sections, obtained by pulling-back the one coming from \(L\) via \(F\): they are à priori well-defined on \(\P^{p} \setminus I_{F}\), but they extend by Riemann's extension theorem.

By invoking once again Riemann's extension theorem, the pull-back section \(F^{*}P\) extends to a non-zero element of \(H^{0}(\P^{p}, E_{p,k,w}\P^{p}\otimes F^{*}L^{-1})\). This contradicts Lemma \ref{lemma: lem3}, and finishes the proof.
\end{proof}

\ack{I would like to thank Erwan Rousseau for his support, his comments and for the helpful discussions we had. I also would like to thank Jorge Vitorio Pereira for his useful comments and advices.}

\bibliographystyle{alpha}
\bibliography{Wronskians}

\end{document}